\documentclass[11pt,reqno]{amsart}
\usepackage{diagrams}
\usepackage{amssymb}
\usepackage[pagewise]{lineno}
\numberwithin{equation}{section}

\theoremstyle{definition}
\newtheorem{thm}{Theorem}[section]
\newtheorem{cor}[thm]{Corollary}
\newtheorem{lem}[thm]{Lemma}
\newtheorem{exa}[thm]{Example}
\newtheorem{prop}[thm]{Proposition}
\newtheorem{defi}[thm]{Definition}
\newtheorem{rem}[thm]{Remark}
\newtheorem{note}[thm]{Notation}
\newtheorem{para}[thm]{}

\newtheorem{se}[thm]{Setup}
\newtheorem{nota}[thm]{Notations/Remark}

\DeclareMathOperator{\NL}{\mathrm{NL}}

\DeclareMathOperator{\red}{\mathrm{red}}

\DeclareMathOperator{\Hc}{\mathcal{H}om}

\DeclareMathOperator{\Ima}{\mathrm{Im}}

\DeclareMathOperator{\op}{\mathcal{O}_{\mathbb{P}^3}(d)}

\DeclareMathOperator{\p3}{\mathbb{P}^3}

\DeclareMathOperator{\pr}{\mathrm{pr}}

\DeclareMathOperator{\Spec}{\mathrm{Spec}}

\DeclareMathOperator{\N}{\mathcal{N}}
\DeclareMathOperator{\T}{\mathcal{T}}

\DeclareMathOperator{\I}{\mathcal{I}}
\DeclareMathOperator{\mo}{\mathcal{O}}

\newcommand{\mr}[1]{\mathrm{#1}}
\newcommand{\mb}[1]{\mathbb{#1}}
\newcommand{\mc}[1]{\mathcal{#1}}
\newcommand{\ov}[1]{\overline{#1}}
\newcommand{\up}[3]{\Upsilon^{#1}_{#2 \le #3}}

\begin{document}

\title[Singularities of the Hilbert scheme]
{Singularities of the Hilbert scheme of non-reduced curves}

\author[A. Dan]{Ananyo Dan}

\address{BCAM - Basque Centre for Applied Mathematics, Alameda de Mazarredo 14,
48009 Bilbao, Spain}

\email{adan@bcamath.org}

\subjclass[2010]{$14$C$05$, $13$D$10$, $14$D$15$, $14$D$07$, $32$G$20$}





\date{\today}

\begin{abstract}
 In this article, we study the Hilbert scheme of generically non-reduced curves in $\p3$.
 We prove the existence of generically non-reduced curves in $\p3$ for which 
 there exist infinitesimal deformations of the curve that do not induce deformations of the associated reduced scheme.
 We show that such infinitesimal deformations contribute to the non-reducedness of the corresponding 
 Hilbert scheme. We introduce the notion of extension of curves and prove that such infinitesimal deformations (such that the associated reduced scheme
 does not deform) are inherited by the extended curve. Finally, we give examples of extension of curves.
 \end{abstract}

\maketitle

\section{Introduction}
Throughout the article the underlying field is $\mb{C}$.
By a \emph{curve} (resp. \emph{surface}) we will mean pure one (resp. two) dimensional closed sub-scheme of $\p3$ and by an \emph{irreducible component of a scheme} we mean the definition in  \cite[\S $2.4.2$]{liu}.
A classical topic in deformation theory is the study of non-reduced (irreducible) components of Hilbert schemes of smooth curves in $\p3$. 
Mumford in \cite{Mu1} was the first to produce an example of a generically non-reduced component of the Hilbert scheme of smooth curves contained in cubic surfaces in $\p3$.
This was generalized to the case of smooth curves contained in surfaces of degree less than $6$  by Kleppe in  \cite{kle2, k1}
and by Kleppe and Ottem in \cite{kle2012}. It was further extended to the case of smooth curves contained in \emph{any} degree surfaces in $\p3$ (see \cite{dan7}).
In almost all these cases, the generic element of the component (of the Hilbert scheme) corresponds to a smooth curve in $\p3$ which belongs to a base point free linear system of a non-reduced curve
twisted by a multiple of a hyperplane section of a smooth surface containing it. For example, the smooth curve studied by Mumford belongs to a linear system $|2l+4H_X|$ on a cubic surface $X$,
where $l$ is one of the $27$ lines on $X$ and $H_X$ is a hyperplane section of $X$. It is natural to ask, if the Hilbert scheme 
of the corresponding non-reduced curve (without the twist by a multiple of a hyperplane section), is also non-reduced? If so, is there a relation between the singularities of the two 
Hilbert schemes (one with and the other without the twist by a multiple of the hyperplane section)? 
Unfortunately, the techniques in the literature cited above, fail if the generic element of the irreducible 
component of the Hilbert scheme is a non-reduced curve. In particular, the 
standard techniques that overlap over the literatures, involve explicit computations of the cohomology 
groups of the normal sheaves of the smooth curves. If the curve is non-reduced, these computations misbehave 
partly because several results on the vanishing of such cohomology groups fail over non-reduced curves. Furthermore, 
the Riemann-Roch theorem and Serre duality on locally free sheaves over non-reduced curves is considerably more complicated. 

In this article we study non-reduced components of Hilbert schemes of non-reduced curves. 
We use tools from Hodge theory to bypass the difficulties mentioned before. One of the advantages of our approach is that 
we are able to reformulate the problem of non-reducedness at a point of the Hilbert scheme, to certain deformation theoretic properties of the curve corresponding to 
this point. More precisely, we prove under mild assumptions, if there exist first order deformations of the curve such that the associated reduced curve does not deform,
then the Hilbert scheme is non-reduced at the point corresponding to the curve (Theorem \ref{phd05}).
Furthermore, we introduce the notion of ``extension of curves'', which is a generalization of twisting by a multiple of a hyperplane section.
We show that given a non-reduced curve $C$, if there exist infinitesimal deformations of $C$ such that the associated reduced curve $C_{\red}$ does not deform, 
then the infinitesimal deformations of any extension of $C$ also has the same property (Proposition \ref{phd06}). 
Thereby, we can immediately conclude the non-reducedness of the 
Hilbert scheme corresponding to the extended curve, without the need for the usual tedious numerical computations.


One of the first results on non-reduced components of Hilbert schemes of curves parametrizing generically non-reduced curves is due to Martin-Deschamps and 
Perrin in \cite{mar1}. In their article, they prove the existence of generically non-reduced components parametrizing extremal, generically non-reduced curves in $\p3$ (see \cite[Definition $0.1$]{mar1} for a definition).
Our first goal is to study these components. Fix integers $a>0$ and $d>a+6$. 
Let $l$ be a line and $C_2$ a smooth coplanar curve of degree $a$ (meaning $C_2$ lies on a plane containing $l$).
Let $X$ be a smooth degree $d$ surface containing $l$ and $C_2$. Denote by $P$ the Hilbert polynomial of the effective divisor 
$2l+C_2$ in $X$. Martin-Deschamps and Perrin proved that there exists an irreducible component $L$ of the Hilbert scheme of curves with Hilbert polynomial $P$, such that the 
generic element of $L$ correspond to an effective divisor $C_g$ in a smooth degree $d$ surface and is of the form $C_g:=2l_g+C_{2,g}$, where $l_g$ is a line 
and $C_{2,g}$ is a degree $a$ coplanar curve (see Theorem \ref{a84}).
We prove that:
\begin{thm}\label{a84e}
For a general point on $L$, the corresponding curve $C_g$ satisfies the property that there exists a {first order} infinitesimal deformation of $C_g$ in $\p3$ which does not induce a deformation of 
the associated reduced scheme $C_{g,\red}$.
 \end{thm}
See Theorem \ref{tr1}, Corollary \ref{phe44} and Remark \ref{phe14} for a more general statement.

We then introduce the notion of extension of curves. Roughly, an \emph{extension of an effective divisor} $C$ contained in a smooth surface $X$
is an effective divisor $D$ in $X$ such that $D-C$ is an effective divisor which intersects $C$ properly and for any first 
order infinitesimal deformation of $C$, there exists a first order deformation of $D$ 
containing it (see Definition \ref{pr5}). We prove,
\begin{thm}\label{desc1}
 Let $D$ be an extension of $C_g$ ($C_g$ as in Theorem \ref{a84e}), $P_D$ the Hilbert polynomial of $D$ and $L_D$ an irreducible component of 
 the Hilbert scheme of curves with Hilbert polynomial $P_D$, containing the point corresponding to $D$. Suppose $\deg(D) \le d-4$ and for a general element $D_g \in L_D$, the Hilbert polynomial of
 $D_{g,\red}$ is the same as that of $D_{\red}$. Suppose further that $D_g$ is contained in a smooth degree $d$ surface. If $D$ is $d$-regular (in the sense of Castelnuovo-Mumford regularity),
 then $L_D$ is non-reduced at the point corresponding to $D$. 
 \end{thm}
See Theorems \ref{phd05} and \ref{phd20} for more general results. See Examples \ref{de-23} and  \ref{de-21} for examples of extensions of curves and Theorem \ref{de-15} for a general criterion 
for extending curves.

Finally, we give an enumerative interpretation to the notion of extension of curves.
 Let $S, T \subset \p3$ be two disjoint sets of reduced, closed points in $\p3$.
Fix an integer $g \ge 0$. 
It is a classical question whether there exists a family of 
curves of genus $g$ passing through $S$, while avoiding the points in $T$.
This question is widely open. In this article, we consider an infinitesimal version of this question. In particular we ask,
does there exist a curve $E \subset \p3$ passing through $S$, such that for each $p \in S$ there exist a first order infinitesimal deformation of $E$ 
passing through $S\backslash\{p\}$ and avoiding $p$? If $E$ satisfies the property in the question, then we say that $E$ is \emph{first order} $S$-\emph{free}.
We then prove:
\begin{thm}[See Theorem \ref{de-18}]
 If the curve  $E$ is first order $S$-free, then for any curve $C$ with $C.E \subset S$, we have $C \cup E$ is a (simple) extension of $C$. 
\end{thm}

We now fix the notations and conventions we will use throughout the article.
\begin{note}
 For the rest of this article, a \emph{surface} will always mean a surface in $\p3$ and a \emph{curve} will mean a closed sub-scheme of $\p3$ of pure dimension $1$. Note a curve need not be reduced.
 Given a closed sub-scheme $Y$ of $\p3$, we denote by $Y_{\red}$ the associated reduced scheme. We denote by $\I_Y$ (resp. $I(Y)$) the ideal sheaf (resp. the ideal $\oplus_{n \in \mb{Z}} \Gamma(\I_Y(n))$) 
 of $Y$ in $\p3$. The ideal $I(Y)$ has a natural grading. Denote by $I_k(Y)$ the $k$-th graded piece of the ideal. Denote by $Q_d$ the Hilbert polynomial of a degree $d$ hypersurface in $\p3$. 
 Given a morphism $f:\mc{X} \to S$ and a sub-scheme $S_0 \subset S$, denote by $\mc{X}_{S_0}:=f^{-1}(S_0)$ the fiber over $S_0$.
\end{note}

{\emph{Acknowledgements}:} This research is part of the author's PhD thesis, under the supervision of Remke Kloosterman. 
I would like to thank him and Inder Kaur for reading the article and useful feedbacks. I thank Catriona Maclean
for suggesting the terminology of ``extension of curves''. I also thank Prof. J. O. Kleppe for showing interest 
in the article and for his useful remarks.
 During my PhD I was funded by the DFG Grant KL-2242/2-1.
The author is currently supported by ERCEA Consolidator Grant $615655$-NMST and also
by the Basque Government through the BERC $2014-2017$ program and by Spanish
Ministry of Economy and Competitiveness MINECO: BCAM Severo Ochoa
excellence accreditation SEV-$2013-0323$.

\section{Preliminaries on Hodge theory and flag Hilbert schemes}

We briefly recall the basic definition of flag Hilbert schemes and its tangent space in the setup we will use in this article. 
See \cite[\S $4.5$]{S1} for the general statements on this topic.

\begin{defi}
Given an $m$-tuple of numerical polynomials $\mathcal{P}(t)=(P_1(t),P_2(t),...,P_m(t))$, we define the  contravariant functor, called the \emph{Hilbert flag functor} 
relative to $\mathcal{P}(t)$,
\begin{eqnarray*}
FH_{\mathcal{P}(t)}:(\mbox{schemes}/\mb{C}) &\to& (\mbox{sets})\\
S &\mapsto& \left\{\begin{tabular}{l|l}
$(X_1,X_2,...,X_m)$& $X_1 \subset X_2 \subset ... \subset X_m \subset \mathbb{P}^3_S, \, X_i \mbox{ are } S-\mbox{flat}$\\
& $ \mbox{with fiberwise Hilbert polynomial } P_i(t)$
\end{tabular}\right\}
\end{eqnarray*}
\end{defi}

\begin{thm}\label{hf1}
 The functor  $FH_{\mathcal{P}(t)}$ is representable by a projective scheme,
denoted $H_{\mc{P}(t)}$, called \emph{the flag Hilbert scheme associated to } $\mc{P}(t)$.
In the case $m=1$, the tangent space $T_oH_{P_1}$ at a point $o \in H_{P_1}$, corresponding to a closed sub-scheme $X_1$ of $\p3$, is isomorphic
to $H^0(\N_{X_1|\p3})$.
In the case $m=2$, the tangent space $T_oH_{P_1,P_2}$ at a point $o \in H_{P_1,P_2}$, corresponding to a pair $(X_1,X_2)$ of closed sub-scheme of $\p3$, 
fits into the following Cartesian diagram:
\begin{equation}\label{dia2}
\begin{diagram}
&T_oH_{P_1,P_2} &\rTo^{\pr_2} &H^0(\N_{X_2|\mathbb{P}^3}) & \\
&\dTo^{\pr_1} & \square &\dTo^{\rho} & \\
&H^0(\N_{X_1|\mathbb{P}^3}) &\rTo^{\beta} &H^0(\N_{X_2|\mathbb{P}^3} \otimes \mathcal{O}_{X_1})
\end{diagram}
\end{equation} 
where $\rho$ is the canonical restriction and $\beta$ is obtained by applying $\Hc(-,\mo_{X_1})$ to the inclusion $\I_{X_2} \hookrightarrow \I_{X_1}$, followed by the global section functor.
\end{thm}


\begin{proof}
 See \cite[\S $4.5$]{S1} for a proof of the theorem.
\end{proof}

We will work in the following setup.

\begin{se}
 Let $X$ be a smooth degree $d$ surface in $\p3$, $C, E$ effective divisors on $X$ satisfying $E \le C$.
 Denote by $P_C$ (resp. $P_E$) the Hilbert polynomial of $C$ (resp. $E$).
 \end{se}

\begin{para}
Recall the following short exact sequence of normal sheaves:
 \begin{equation}\label{sh2}
  0 \to \N_{C|X} \to \N_{C|\p3} \to \N_{X|\p3}  \otimes \mo_C \to 0
 \end{equation}
 which arises from the short exact sequence
 \begin{equation}\label{ex0c}
 0 \to \I_X \to \I_C \to \mo_X(-C) \to 0
 \end{equation}
after applying $\Hc_{\p3}(-,\mo_C)$. We get a similar short exact sequence after replacing $C$ by $E$.
\end{para}

 The following diagram relates first order deformations of $E$ to that of $C$. This diagram plays an important role throughout this article.
 
 \begin{rem}
 We have the following commutative diagram:
 \begin{equation}\label{phe11}
  \begin{diagram}
                & &T_{(E,C)}H_{P_{E},P_C} &\rTo^{\pr_1} &H^0(\N_{E|\p3})\\
                & &\dTo^{\pr_2} &\square &\dTo_{\up{6}{E}{C}}\\
               T_{(C,X)}H_{P_C,Q_d}& \rTo^{\pr_1} &H^0(\N_{C|\p3})&\rTo^{\up{5}{E}{C}}&H^0(\N_{C|\p3}  \otimes \mo_{E})\\
               \dTo^{\pr_2} &\square &\dTo^{\beta_C}&\circlearrowleft&\dTo_{\up{1}{E}{C}}\\
                H^0(\N_{X|\p3}) &\rTo^{\rho_{_C}}&H^0(\N_{X|\p3} \otimes \mo_{C})&\rTo^{\up{2}{E}{C}}&H^0(\N_{X|\p3}  \otimes \mo_{E})
               \end{diagram}
  \end{equation}
 where $\rho_{_C}$, $\up{2}{E}{C} \mbox{ and } \up{5}{E}{C}$ are restriction morphisms, $\up{6}{E}{C}$ is obtained by applying the functor $\Hc_{\p3}(-,\mo_E)$ followed by the global section 
 functor to the natural morphism $\I_{C} \hookrightarrow \I_E$, $\up{1}{E}{C}$ (resp. $\beta_C$) arises from the short exact sequence \eqref{sh2}.
 The two Cartesian diagrams follow from the theory of flag Hilbert schemes (see Theorem \ref{hf1}).
 \end{rem}

 Let us now study some of the basic properties of this diagram. These will be used throughout this article. We first recall the following useful lemma:
 
 
  \begin{lem}[{\cite[Lemma $3.6$]{D3}}]\label{a4e}
Suppose that $d \ge \max\{\deg(C)+2,5\}.$
Then $h^0(\N_{C|X})=0.$ In particular, $\dim |C|=0$ where $|C|$ is the linear system associated to $C$.
\end{lem}

   \begin{lem}\label{nr13}
  Let $d \ge \deg(C)+4$. We have the following:
     \begin{enumerate}
   \item $\up{1}{E}{C} \circ \up{6}{E}{C}$ is the same as the morphism $\beta_{E}$ from $H^0(\N_{E|\p3})$
    to  $H^0(\N_{X|\p3}\otimes \mo_{E})$ arising from the short exact sequence (\ref{sh2}) (replace $C$ by $E$),
   \item $\up{2}{E}{C} \circ \rho_{_C}$ is the same as the natural restriction morphism $\rho_{_E}$
  from $H^0(\N_{X|\p3})$ to $H^0(\N_{X|\p3} \otimes \mo_{E})$,
  \item  the morphism $\up{6}{E}{C}$ is injective,
  \item  the fiber of the projection morphism from $H_{P_E,P_C}$ to $H_{P_C}$ over
   any closed reduced point is a discrete set of closed reduced points,
  \item the morphism $\beta_C$ is injective,
  \item the morphism $\up{1}{C_{\red}}{C}$ is injective. Moreover, for any reduced curve $E \le C_{\red}$, we have
    $H^0(\N_{C|X} \otimes \mo_E)=0$ i.e., $\up{1}{E}{C}$ is injective.
   \end{enumerate}          
    \end{lem}
               
   \begin{proof}
   The proof of $(1)$ and $(2)$ follows directly from definition.
   It follows from $(1)$ that the kernel of $\up{1}{E}{C} \circ \up{6}{E}{C}$ is $H^0(\N_{E|X})$.
    Lemma \ref{a4e} implies that $h^0(\N_{E|X})=0$. Therefore, $\up{1}{E}{C} \circ \up{6}{E}{C}$ is injective, hence $\up{6}{E}{C}$ is injective. 
    This proves $(3)$. 
    
     Note that the tangent space at $(E,C)$ to the fiber of $H_{P_E,P_C}$ is the
   kernel of the natural projection morphism $T_{(E,C)}H_{P_E,P_C} \to T_{C}H_{P_C}$. 
   It follows from the  diagram \eqref{phe11} this is isomorphic to $\ker \up{6}{E}{C}$.
   But $(3)$   implies $\ker \up{6}{E}{C}=0$. Hence, the fiber of the projection morphism from $H_{P_E,P_C}$ to $H_{P_C}$ over
   any closed reduced point is zero dimensional and reduced. This proves $(4)$.
   
   Since the morphism $\beta_C$ is induced by the short exact sequence \eqref{sh2},
   the kernel of $\beta_C$ is $H^0(\N_{C|X})$, the vanishing of which follows from Lemma \ref{a4e}. This proves $(5)$.
   
      Consider the long exact sequence of the following short exact sequence:
    \[0 \to \N_{C|X} \otimes \mo_{C_{\red}} \to \N_{C|\p3} \otimes \mo_{C_{\red}} \to \N_{X|\p3} \otimes \mo_{C_{\red}} \to 0\]
    obtained by pulling back to $C_{\red}$ the short exact sequence (\ref{sh2}). Then the kernel of the map $\up{1}{C_{\red}}{C}$ is $H^0(\N_{C|X} \otimes \mo_{C_{\red}})$.
    
    We show that the degree of the line bundle $\N_{C|X} \otimes \mo_{C_{\red}}$ restricted to each of the components is negative. 
    This would mean there do not exist global sections on any of the irreducible components, hence not on $C_{\red}$. 
    We can write $C=\sum_{i=1}^r a_iC_i$ in $\mr{Div}(X)$ with $C_i \not= C_j$ for $i \not= j$ and $C_i$ integral. It suffices to prove this for $C_1$.
    Note that \[C.C_1= a_1C_1^2+\sum_{i=2}^r a_iC_iC_1 \le a_1(2\rho_a(C_1)-2-(d-4)\deg(C_1))+\sum_{i=2}^r a_i\deg(C_i)\deg(C_1)\]
    which follows from the fact that $C_i.C_1 \le \deg(C_1)\deg(C_i)$ for $i \not= 1$ and the adjunction formula applied to $C_1^2$.
    Using the degree assumption on $d$ and the bound on the arithmetic genus of an integral curve in $\p3$ we then get
    {\small \[C.C_1 \le \left(\sum_{i=2}^r a_i\deg(C_i)\deg(C_1)\right)+a_1\left((\deg(C_1)-1)(\deg(C_1)-2)-2-\deg(C_1)\left(\sum_{i=1}^r a_i\deg(C_i)\right)\right)\]}
    which is clearly less than zero
    since $a_i \ge 1$ for all $i$. Hence, $h^0(\N_{C|X} \otimes \mo_{C_{\red}})=0$. This proves the first part of $(6)$.
    
    The second part is a direct consequence of the proof of the first part. In particular, we see that for any irreducible component of $C_{\red}$
    there does not exist a global section of the restriction of $\N_{C|X}$ to this component. So, for any $E \le C_{\red}$, $h^0(\N_{C|X} \otimes \mo_{E})=0$.
    
    Pulling back the short exact sequence (\ref{sh2}) to $E$, one can 
    observe as before that $\ker \up{1}{E}{C}=h^0(\N_{C|X} \otimes \mo_{E})=0$. Therefore, $\up{1}{E}{C}$ is injective. This proves $(6)$ and hence the lemma.
  \end{proof}

We now recall the basics of Hodge theory, again restricting to the situation relevant for this article. See \cite[\S $5$]{v5} for a detailed study 
of this subject.

\begin{note}
Denote by $U_d \subseteq \mathbb{P}(H^0(\mathbb{P}^3, \op))$ 
the open sub-scheme parametrizing smooth projective hypersurfaces in $\mathbb{P}^3$ of degree $d$. Denote by $Q_d$ the Hilbert polynomial of degree $d$ surfaces in $\p3$. 
Let \[\pi:\mathcal{X} \to U_d\] be the corresponding universal family. For a given $u \in U_d$, denote by $X_u$ the surface $X_u:=\pi^{-1}(u)$.
Fix a closed point $o \in U_d$, denote by $X:=X_o$ and consider a simply connected neighbourhood $U$ of $o$ in $U_d$  (in the analytic topology).
\end{note}

\begin{defi}
 As $U$ is simply connected, $\pi|_{\pi^{-1}(U)}$ induces a variation of Hodge structure $(\mathcal{H}^2, \nabla)$ on $U$ where $\mathcal{H}^2:=R^2\pi_*\mathbb{Z} \otimes 
\mathcal{O}_U$ and $\nabla$ is the Gauss-Manin connection. Note that $\mathcal{H}^2$ defines a local system on $U$ whose fiber over 
a point $u \in U$ is $H^2(X_u,\mathbb{C})$. Consider a non-zero element $\gamma_0 \in H^2(X,\mathbb{Z}) \cap H^{1,1}(X,\mathbb{C})$
such that $\gamma_0  \not= c_1(\mathcal{O}_{X}(k))$ for $k \in \mathbb{Z}_{>0}$. 
This defines a flat section $\gamma \in \Gamma(U,\mathcal{H}^2)$ which takes the value $\gamma_0$ at the point $o \in U$. 
Recall, there exists a sub-bundle $F^2\mc{H}^2 \subset \mc{H}^2$, which for any $u \in U$, can be identified with
the Hodge filtration $F^2H^2(X_u,\mb{C}) \subset H^2(X_u,\mb{C})$ (see \cite[\S $10.2.1$]{v4}).
Let $\overline{\gamma}$ be the image
of $\gamma$ in $\mathcal{H}^2/F^2\mathcal{H}^2$. The \emph{Hodge locus}, denoted $\NL(\gamma_0)$ is then defined as
\[ \NL(\gamma_0):=\{u \in U | \overline{\gamma}_u=0\},\]
where $\overline{\gamma}_u$ denotes the value at $u$ of the section $\overline{\gamma}$. One can check that 
the Hodge locus $\NL(\gamma_0)$ does not depend on the choice of the section $\gamma$.
\end{defi}


 \begin{lem}[{\cite[Lemma $5.13$]{v5}}]
 There is a natural analytic scheme structure on $\ov{\NL(\gamma_0)}$ (closure in $U_d$ in the Zariski topology).
\end{lem}

\begin{rem} 
We now discuss the tangent space to the Hodge locus $\NL(\gamma_0)$.
We know that the tangent space to $U$ at $o$, $T_oU$ is isomorphic to $H^0(\N_{X|\p3})$.
This is because $U$ is an open sub-scheme of the Hilbert scheme $H_{Q_d}$, the tangent space of which at the point $o$ is simply
$H^0(\N_{X|\p3})$. 
Given the variation of Hodge structure above, we have (by Griffith's transversality) the differential map:
\[\overline{\nabla}:H^{1,1}(X) \to \mathrm{Hom}(T_oU,H^2(X,\mathcal{O}_X))\]induced by the Gauss-Manin connection.
Given $\gamma_0 \in H^{1,1}(X)$ this induces a morphism, denoted $\overline{\nabla}(\gamma_0)$, from $T_oU$ to $H^2(\mo_X)$.
\end{rem}

\begin{lem}[{\cite[Lemma $5.16$]{v5}}]\label{tan1}
The tangent space at $o$ to $\NL(\gamma_0)$, denoted $T_o\NL(\gamma_0)$, equals $\ker(\overline{\nabla}(\gamma_0))$.
\end{lem}

We now discuss the semi-regularity map first introduced by Kodaira-Spencer in the case of divisors and then generalized to 
any local complete intersection sub-schemes by Bloch. The purpose of this map is to study certain aspects 
of the variational Hodge conjecture. 
We consider the cohomology class $\gamma$ of a curve $C$ in a smooth degree $d$ surface $X$ in $\p3$.
We see that the differential map $\ov{\nabla}(\gamma)$ factors through the semi-regularity map (see Theorem \ref{hf12}).
Using this description, we are able to identify a subspace of $T_o\NL(\gamma)$
parametrizing first order deformations $X'$ of $X$ such that $C$ deforms to a local complete intersection sub-scheme of $X'$ (see Corollary \ref{hf12c}).
 
We start with the definition of a semi-regular curve.
\begin{defi}
 Let $X$ be a smooth surface and $C \subset X$ a curve in $X$. 
Consider the natural short exact sequence:
 \begin{equation}\label{sh1}
 0 \to \mo_X \to \mo_X(C) \to \N_{C|X} \to 0.
                 \end{equation}                       
 The \emph{semi-regularity map} $\pi_C$ is the boundary map from $H^1(\N_{C|X})$ to $H^2(\mo_X)$, coming from this short exact sequence.
 We say that $C$ is \emph{semi-regular in } $X$ if $\pi_C$ is injective. 
\end{defi}



The following lemma gives a criterion for $C$ to be semi-regular.  
               
\begin{lem}[{\cite[Lemma $5.2$]{dan7}}]\label{hf11}
Let $C$ be a reduced curve and $X$ a smooth degree $d$ surface containing $C$. Then $H^1(\mo_X(-C)(k))=0$ for all $k \ge \deg(C)$. In particular, 
if $d \ge \deg(C)+4$, then $h^1(\mathcal{O}_X(C))=0$, hence $C$ is semi-regular.
\end{lem}

We then have the following results on the tangent space of the Hodge locus.

\begin{thm}[{\cite[Theorem $3.6$]{dan7}}]\label{hf12}
Let $X$ be a smooth degree $d$ surface and $C$ a curve in $X$. Let $\gamma=[C] \in H^{1,1}(X,\mb{Z})$, be the cohomology class of $C$.
 We then have the following commutative diagram
 \[\begin{diagram}
  & &  & &T_{(C,X)}H_{P,Q_d}&\rTo&H^0(X,\N_{X|\p3})&\rTo^{\ov{\nabla}(\gamma)}&H^2(X,\mo_X)\\
  & &  & &\dTo&\square&\dTo^{\rho_{_C}}&\circlearrowright&\uTo^{\pi_C}& \\
    0&\rTo&H^0(C,\N_{C|X})&\rTo^{\phi_C}&H^0(C,\N_{C|\p3})&\rTo^{\beta_C}&H^0(C,\N_{X|\p3} \otimes \mo_C)&\rTo^{\delta_C}&H^1(C,\N_{C|X}) 
       \end{diagram}\]
where the horizontal bottom row is the exact sequence associated to (\ref{sh2}), $\pi_C$ is the semi-regularity map and $\rho_{_C}$ is the natural
pull-back morphism.
\end{thm}

\begin{cor}[{\cite[Corollary $3.7$]{dan7}}]\label{hf12c}
 Denote by $P$ the Hilbert polynomial of $C$. 
Then the tangent space $T_o\NL(\gamma)$ to $\NL(\gamma)$ at the point $o$ corresponding to $X$, satisfies the following:
\[T_o(\NL(\gamma)) \supset \rho_{_C}^{-1}(\Ima \beta_C)= \pr_2T_{(C,X)}H_{P,Q_d}.\]
Furthermore if $C$ is semi-regular, then we have equality
$T_o(\NL(\gamma))=\pr_2T_{(C,X)}H_{P,Q_d}$.
\end{cor}

 The following corollary tells us for which first order deformations of $X$, the cohomology classes 
 $[C]$ and $[E]$ remain Hodge.

               \begin{cor}\label{pr13}
              For all $t \in (\rho_{_C})^{-1}(\Ima \beta_C)$ (resp. $(\up{2}{E}{C} \circ \rho_{_C})^{-1}(\Ima \up{1}{E}{C} \circ \up{6}{E}{C})$), we have 
                 \[\overline{\nabla}([C])(t)=0 \, \, (\mbox{ resp. } \ov{\nabla}([E])(t)=0).\]
   Furthermore,  if $C$ is reduced and $\deg(X) \ge \deg(C)+4$, then $\overline{\nabla}([C])(t)=0$ if and only if $t \in (\rho_{_C})^{-1}(\Ima \beta_C)$.
               \end{cor}
               
               \begin{proof}
                Lemma \ref{nr13} implies $\up{2}{E}{C} \circ \rho_{_C}=\rho_{_E}$ and $\up{1}{E}{C} \circ \up{6}{E}{C}=\beta_{E}$. The first part of the corollary then follows from Corollary \ref{hf12c}.
                By Lemma \ref{hf11}, $C$ is semi-regular. Corollary \ref{hf12c} implies that $t \in (\rho_{_C})^{-1}(\Ima \beta_C)$ if and only if $\overline{\nabla}([C])(t)=0$. This proves the corollary.
               \end{proof}

\section{Hilbert schemes of non-reduced curves}

Fix the Hilbert polynomial $P$ of a curve in $\p3$. 
The aim of this section is to study certain topological aspects of the corresponding Hilbert scheme of curves $H_P$. 
Let $L$ be an irreducible component of $H_P$. We first prove that there exists a Hilbert polynomial $P_r$ such that 
every curve $D \in L$ contains a sub-curve $D' \subset D$ with Hilbert polynomial $P_r$ and for a general curve $D \in L$ (i.e., $D$ lies outside finitely many proper closed subsets of $L$),
$D_{\red}$ is of Hilbert polynomial $P_r$. 

\begin{prop}\label{phe8}
Let  $L$ be an irreducible component of $H_P$. There exists a Hilbert polynomial $P_r$ of a curve in $\p3$ 
and an irreducible component $L_r$ of $H_{P_r,P}$ such that:
\begin{enumerate}
 \item the natural projection $\pr_2$ maps $L_r$ surjectively (as topological spaces) to $L$,
 \item for any $u \in L_r$, the corresponding pair $(\mc{C}'_u,\mc{C}_u)$ satisfies $\mc{C}'_{u,\red}=\mc{C}_{u,\red}$,
 \item there exists a non-empty open subset $U_r$ of $L_r$ such that for all $u \in U_r$, the corresponding pair $(\mc{C}'_u,\mc{C}_u)$
 satisfies $\mc{C}'_u \cong \mc{C}'_{u,\red}$ i.e., $\mc{C}'_u$ is reduced.
\end{enumerate}
Furthermore, the choice of $P_r$ and $L_r$ satisfying the three properties is \emph{unique}.
\end{prop}

\begin{proof}
 Let $\mc{C}_L \xrightarrow{\pi} L$ be the universal family over $L$. There exists a morphism $\ov{\pi}:\mc{C}_{L,\red} \xrightarrow{\bar{\pi}} L_{\red}$ such that the following diagram
 is commutative,
 \[\begin{diagram}
    \mc{C}_{L,\red}&\rInto&\mc{C}_L\\
    \dTo^{\ov{\pi}}&\circlearrowleft&\dTo_{\pi}\\
    L_{\red}&\rInto&L
   \end{diagram}\]
   By \cite[Theorem $6.9.1$]{ega42}, there exists a nonempty open set $U \subset L_{\red}$ such that \[\ov{\pi}|_{\ov{\pi}^{-1}(U)}:\ov{\pi}^{-1}(U) \to U\] is flat.
   Then \cite[Theorem III.$9.9$]{R1} implies that every fiber of $\ov{\pi}|_{\ov{\pi}^{-1}(U)}$ has the same Hilbert polynomials, say $P_r$. 
   
  Denote by $\mc{C}'_U:=\ov{\pi}^{-1}(U)$ and $\mc{C}_U:=\pi^{-1}(U)$. Clearly, for all $u \in U$, $\mc{C}'_{u,\red}=\mc{C}_{u,\red}$. Denote by $\ov{\eta}_U$  the geometric 
  generic point of $U$ and $\mc{C}'_{\ov{\eta}_U}:=\ov{\pi}^{-1}(\ov{\eta}_U)$. Note that $\mc{C}'_{\ov{\eta}_U}$ is reduced (reducedness is preserved under flat base change in char. $0$).
  Since $\mc{C}'_U$ and $\mc{C}_U$ are flat over $U$ with fiberwise Hilbert polynomials $P_r$ and $P$, respectively, there exists an unique morphism 
  $\phi_U:U \to H_{P_r,P}$ such that the pull-back of the universal family over $H_{P_r,P}$ to $U$ is isomorphic to the pair $(\mc{C}'_U,\mc{C}_U)$.
  Let $L_r$ be an irreducible component of $H_{P_r,P}$ containing $\phi_U(\ov{\eta}_U)$. The uniqueness of $\phi_U$ implies that $\pr_2 \circ \phi_U$ is identify on $U$. In particular,
  $\pr_2 \circ \phi_U(\ov{\eta}_U)=\ov{\eta}_U$. Since $\pr_2$ is proper, we have that $\pr_2(L_r)$ is an irreducible, closed sub-scheme of $H_{P,\red}$ containing $\ov{\eta}_U$. As $\ov{\eta}_U$ is the 
  geometric generic point of $L_{\red}$, we conclude that $\pr_2(L_r)_{\red}=L_{\red}$ i.e., condition $(1)$ is satisfied.
  
  Note that geometric reducedness is an open property (see \cite[Theorem $12.2.4$]{ega43}). Hence, there exists a non-empty open neighbourhood $V \subset L_r$ containing $\phi_U(\ov{\eta}_U)$ such that for all
  $v \in V$, the corresponding pair $(\mc{C}'_v,\mc{C}_v)$ satisfies: $\mc{C}'_v$ is geometrically reduced i.e., condition $(3)$ is satisfied.
  We now prove condition $(2)$. Denote by $(\mc{C}'_{L_r},\mc{C}_{L_r})$  the universal family on $L_r$ and $\pi_r:\mc{C}_{L_r} \to L_r$ the natural morphism.
  Since $\pi_r$  is a flat morphism of finite type between noetherian schemes, \cite[Ex. III.$9.1$]{R1} implies that $\pi_r$ is open. 
   Hence, $\pi_r(\mc{C}_{L_r}-\mc{C}'_{L_r})$ is open. 
   Note that for all $u \in \phi_U^{-1}(V)$, we have 
   \[\mc{C}'_u \cong \mc{C}'_{u,\red} \cong \mc{C}_{u,\red}\]
   i.e., $\mc{C}_{u,\red}$ is of Hilbert polynomial $P_r$. 
   Therefore, for all $v \in \pr_2^{-1}(\phi_U^{-1}(V)) \cap V$, we have 
   \[\mc{C}'_v \subset \mc{C}_{v,\red} \, \mbox{ and the Hilbert polynomial of } \mc{C}_{v,\red} \mbox{ is } P_r\]
   (the first inclusion follows from $\mc{C}'_v$ being reduced).
   This implies for all $v \in \pr_2^{-1}(\phi_U^{-1}(V)) \cap V$, $\mc{C}'_v  \cong \mc{C}_{v,\red}$.
   In other words, $\pr_2^{-1}(\phi_U^{-1}(V)) \cap V \cap \pi_r(\mc{C}_{L_r}-\mc{C}'_{L_r})=\emptyset$. Since $L_r$ is irreducible and $\pr_2^{-1}(\phi_U^{-1}(V)) \cap V \not= \emptyset$,
   we conclude that $\pi_r(\mc{C}_{L_r}-\mc{C}'_{L_r})=\emptyset$. Therefore, for all $v \in L_r$, $\mc{C}'_{v,\red} \cong \mc{C}_{v,\red}$ i.e., 
   condition $(2)$ is satisfied. This proves the first part of the proposition. We now prove uniqueness.
   
   Note that the choice of $P_r$ is unique by property $(3)$, as the Hilbert polynomial of $C_{g,\red}$ for a general curve $C_g \in L$ is $P_r$.
   We now prove the uniqueness of $L_r$. In particular, if $L'_r$ is another irreducible component of 
   $H_{P_r,P}$ which maps surjectively to $L$ and contains a point of 
 the form $(C',C)$ with $C'$ reduced, then $L'_r=L_r$. Indeed, since geometric reducedness is an open property, 
 there exists an open subset $V \subset L'_r$ such that any point $v \in V$ correspond to a pair 
 $(\mc{C}'_v,\mc{C}_v)$ with $\mc{C}'_v$ reduced. Let $U \subset L$ be the open subset consisting of all points
 $u \in L$ such that for the corresponding curve $\mc{C}_u$, $\mc{C}_{u,\red}$ has Hilbert polynomial $P_r$.
 Since $L'_r$ is irreducible, $V \cap \pr_2^{-1}(U)$ is non-empty and for all $v \in V \cap \pr_2^{-1}(U)$,
 $\mc{C}'_v = \mc{C}'_{v,\red} \subset \mc{C}_{v,\red}$ and $\mc{C}'_v$ has the same Hilbert polynomial as $\mc{C}_{v,\red}$. Hence,
 $\mc{C}'_v=\mc{C}_{v,\red}$ for all $v \in V \cap \pr_2^{-1}(U)$. Then $(3)$ implies that $L'_r \cap L_r$ contains a non-empty open 
 subset of $V \cap \pr_2^{-1}(U)$, which is open, dense in both $L_r$ and $L'_r$, namely the set of points $v$ in $L$ and $L'$ of the form $(\mc{C}_{v,\red},\mc{C}_v)$.
 This implies that $L'_r=L_r$. This proves the proposition.
\end{proof}

\begin{defi}
 Let $C$ be a curve in $\p3$. We say that $C$ is $d$-\emph{embeddable} if there exists a smooth degree $d$ surface $X$ in $\p3$ containing $C$.
 The pair $(C,X)$ is called a $d$-\emph{embedded curve}. The \emph{Hilbert polynomial of  the $d$-embedded curve } $(C,X)$ is defined to be
 the Hilbert polynomial of $C$.
\end{defi}

The following lemma will play an important role in the proof of Theorem \ref{phe42} later.

\begin{lem}\label{pa22}
Let $C$ be a $d$-embeddable curve with Hilbert polynomial $P$. 
For any smooth degree $d$ surface $X$ containing
$C_{\red}$ there exists a curve $D \subset X$ with Hilbert polynomial $P$ and $D_{\red}=C_{\red}$.
 For every degree $d$ surface $X$ containing 
$C_{\red}$ there exists a curve $D$ in $X$ such that $D \in H_P$ and $D_{\red}=C_{\red}$.
\end{lem}

\begin{proof}
Let $X$ be a smooth degree $d$ surface containing $C$ and $C=\sum_i a_iC_i$ for $a_i>0$ and $C_i$ integral curves. 
Clearly, $\deg(C)$ does not depend on the surface containing it.
If $Y$ is another smooth degree $d$ surface containing $C_{\red}$, then the arithmetic genus of the divisor $\sum_i a_iC_i$ of $Y$ is the same
as that of $C$ (use the adjunction formula). This proves the lemma.
\end{proof}

  Recall, the following result on Castelnuovo-Mumford regularity:
  \begin{thm}\label{cas1}
   Let $C$ be a reduced curve in $\p3$ of degree $e$. Then $C$ is $e$-regular.
  \end{thm}

  \begin{proof}
   If $C$ is connected, then it is $e$-regular (\cite[Main Theorem]{gi1}). Note that \cite[Theorem $1.8$]{sid} states that for homogeneous ideals $I, J$ in some polynomial ring, 
   the regularity of $I.J$ is at most the sum of the regularity of $I$ and $J$.
   This implies that if $C=C_1 \cup  ... \cup  C_n$ and $C_i$ are the connected components, then the regularity of the ideal of $C$ is at most the sum of the regularity of the 
   ideals of $C_i$ for $i=1,...,n$. The theorem then follows.   
  \end{proof}

  \begin{lem}\label{de-13}
  Let $C$ be a $d$-embeddable curve for some $d \ge \deg(C_{\red})+4$ and $P_r$ be the Hilbert polynomial of $C_{\red}$. Then  $C_{\red}$ is the only sub-curve of 
  $C$ with Hilbert polynomial $P_r$.
  \end{lem}

  \begin{proof}
   Let $X$ be a smooth degree $d$ surface containing $C$. Then $C$ is of the form $\sum_{i=1}^r a_i C_i$ for some $r>0, a_i>0$ and $C_i$ integral curves.
   Let $C'$ be a sub-curve of $C$ with Hilbert polynomial $P_r$. Then $C'=\sum_{i=1}^r (a_i-b_i)C_i$ with $a_i \ge b_i \ge 0$.
   By assumption, $\rho_a(C')=\rho_a(C_{\red})$ and $\deg(C')=\deg(C_{\red})$. By adjunction formula, this implies $C'^2=C_{\red}^2$. Note that $C'^2-C_{\red}^2$ is equal to 
\[\sum\limits_{i,j} ((a_i-b_i)(a_j-b_j)-1)C_iC_j=\left(\sum\limits_{i=1}^r ((a_i-b_i)^2-1)C_i^2\right)+2\left(\sum\limits_{i<j} ((a_i-b_i)(a_j-b_j)-1)C_iC_j\right).\]
 By the adjunction formula and the assumption on $d$, we have \[C_i^2=2\rho_a(C_i)-2-(d-4)\deg(C_i) \le 2\rho_a(C_i)-2-\deg(C_i)\sum\limits_{j=1}^r \deg(C_j).\]
 Note that the arithmetic genus of $C_i$ is bounded above by the arithmetic genus of a complete intersection curve of degree $\deg(C_i)$ i.e., $\rho_a(C_i) \le (\deg(C_i)-1)(\deg(C_i)-2)/2$.
 Hence, \[C_i^2 \le (\deg(C_i)-1)(\deg(C_i)-2)-2-\deg(C_i)\sum\limits_j \deg(C_j)<-\deg(C_i)\sum\limits_{j, j \not= i} \deg(C_j).\]
 Since $C_i.C_j \le \deg(C_i)\deg(C_j)$, we then obtain
 \[C'^2-C_{\red}^2 < -\left(\sum\limits_i \left(((a_i-b_i)^2-1) \deg(C_i)\sum\limits_{j, j \not= i} \deg(C_j)\right)\right)+\]\[+2\left(\sum\limits_{i<j} ((a_i-b_i)(a_j-b_j)-1)\deg(C_i)\deg(C_j)\right).\]
 Since $((a_i-b_i)^2-1)+((a_j-b_j)^2-1)-2((a_i-b_i)(a_j-b_j)-1) \ge 0$, we can immediately conclude that 
  \[C'^2-C_{\red}^2 \le 0,\] with equality if and only if $a_i=b_i+1$ for all $i$ i.e., $C'=C_{\red}$.
 This proves the lemma.
  \end{proof}

\begin{note}\label{de-4}
 Notations as in Proposition \ref{phe8}. Denote by $\mc{N}:=(P_r,P,Q_d)$ the ordered triple of Hilbert polynomials. For $1 \le i<j \le 3$, denote by 
 $\mc{N}_{ij}$ the ordered pair consisting of the $i$-th and the $j$-th coordinate of $\mc{N}$. Similarly, denote by $\mc{N}_i$ the $i$-th coordinate of $\mc{N}$.
 Denote by $(123 \to ij)$ (resp. $(123 \to i)$) the natural projection map from $H_{\mc{N}}$ to $H_{\mc{N}_{ij}}$ (resp. $H_{\mc{N}_i}$). Denote by 
 $(ij \to k)$ the natural projection morphism from $H_{\mc{N}_{ij}}$ to $H_{\mc{N}_k}$. 
\end{note}

\begin{nota}\label{de-1}
Fix an integer $d>0$.
 Let $L \subset H_P$ be an irreducible component of $H_P$ such that a general element of $L$ is a $d$-embeddable curve.
 Let $P_r$ and $L_r$ be as in Proposition \ref{phe8}.
 Then  a general element of $L_r$ is of the form $(C_{g,\red},C_g)$.
 The fiber over the generic point $(C_{g,\red},C_g)$ of $L_r$ to the morphism $(123 \to 12)$ is isomorphic to $\mb{P}(I_d(C_g))$ (i.e.,
 the space of degree $d$ surfaces containing $C_g$).
 Since $\mb{P}(I_d(C_g))$ is irreducible, there exists an \emph{unique} irreducible component $L_{rs}$ of $H_{\mc{N}}$ which maps 
 surjectively (as topological spaces) to $L_r$ via the morphism $(123 \to 12)$.
\end{nota}

We now prove the main theorem of this section. The dimension computation in the theorem plays an important role in the proof of Theorems \ref{tr1} and \ref{phd20}.

\begin{thm}\label{phe42}
 Let $L, L_r$ be as in Notations/Remark \ref{de-1}. Assume that a general element $C_g$ of $L$ is $d$-embeddable for some $d \ge \deg(C_{g})+4$.
 Let $u \in (12 \to 1)(L_r)$ correspond to a reduced, $d$-embeddable curve $\mc{C}'_u$. 
 Let $(\mc{C}'_u,\mc{C}_{u,g})$ be a general element of $(12 \to 1)^{-1}(u) \cap L_r$.
 Then 
 \[\dim ((12 \to 1)^{-1}(u) \cap L_r)=\dim I_d(\mc{C}'_u)-\dim I_d(\mc{C}_{u,g})=\dim I_d(\mc{C}'_u)-\dim I_d(C_g).\]
\end{thm}

\begin{proof}
 Let $L_{rs}$ be the irreducible component of $H_{\mc{N}}$ mentioned in Notations/Remark \ref{de-1}. Since $(123 \to 12)|_{L_{rs}}$ maps 
 surjectively to $L_r$, we have 
 \begin{equation}\label{de-11}
  (12 \to 1)^{-1}(u) \cap L_r=(12 \to 1)|_{L_r}^{-1}(u)=(123 \to 12)|_{L_{rs}}((123 \to 1)|_{L_{rs}}^{-1}(u)).
 \end{equation}
 By Proposition \ref{phe8}, a general element of $L_r$ is of the form $(C_{g,\red},C_g)$. By assumption, $C_g$ is $d$-embeddable.
 Then  Lemma \ref{pa22} implies that every smooth degree $d$ surface containing $C_{g,\red}$ contains a curve $D$ with Hilbert polynomial $P$
 and $D_{\red}=C_{g,\red}$. Since $C_g$ is general, Proposition \ref{phe8} implies that all such pairs $(C_{g,\red},D)$ lie in $L_r$ (use the uniqueness of $L_r$).
 As $L_{rs}$ is the unique irreducible component of $H_{\mc{N}}$ that maps surjectively to $L_r$, one can check that the fiber 
 over all such pairs $(C_{g,\red},D)$, for the morphism $(123 \to 12)|_{L_{rs}}$ parametrizes all degree $d$ surfaces containing $D$.
 This immediately implies that $(123 \to 1)|_{L_{rs}}^{-1}(C_{g,\red})$ parametrizes all degree $d$ surfaces containing $C_{g,\red}$.
 More precisely,
 \[(123 \to 13)((123 \to 1)|_{L_{rs}}^{-1}(C_{g,\red})) \cong \mb{P}(I_d(C_{g,\red})).\]
 By Theorem \ref{cas1}, the Castelnuovo-Mumford regularity of $C_{g,\red}$ is $\deg(C_{g,\red})$. Since $d \ge \deg(C)+4$, we therefore have 
 $\dim \mb{P}(I_d(C_{g,\red}))=P_r(d)-1$ (the Hilbert polynomial of $C_{g,\red}$ is $P_r$).

 Denote by $T_0:=(123 \to 13)(L_{rs})$, $T_1:=(123 \to 1)(L_{rs})$ and $F:=(123 \to 1)|_{L_{rs}}^{-1}(u)$. Then  the induced natural morphism $(13 \to 1)|_{T_0}:T_0 \to T_1$ is dominant.
 Since $C_{g,\red}$ is a general element of $T_1$, the observation above can be rewritten as: the generic fiber of $(13 \to 1)|_{T_0}$ is 
 of dimension $P_r(d)-1$. By the upper semicontinuity of fiber dimension, this implies 
 \[\dim (13 \to 1)|_{T_0}^{-1}(u) \ge P_r(d)-1.\]
 Now, $(13 \to 1)|_{T_0}^{-1}(u)$ is isomorphic to a subspace of $\mb{P}(I_d(\mc{C}'_u))$.
 Since $\dim I_d(\mc{C}'_u)=P_r(d)$ (as before $\mc{C}'_u$ is $d$-regular),
 we conclude that $\dim (13 \to 1)|_{T_0}^{-1}(u) = P_r(d)-1$.
 Since $\mc{C}'_u$ is $d$-embeddable, for a general element $(\mc{C}'_u,X_{u,g})$ in $(13 \to 1)|_{T_0}^{-1}(u)$, we have $X_{u,g}$ is a smooth, degree $d$ surface containing $\mc{C}'_u$.
 Using the adjunction formula, one can check that there are finitely many curves $D$ in $X_{u,g}$ with Hilbert polynomial $P$ and $D_{\red} \cong \mc{C}'_u$. 
  This implies that the generic fiber of \[(123 \to 13)|_F:F \to (13 \to 1)|_{T_0}^{-1}(u)\] is zero dimensional.
  Therefore, \[\dim F=\dim (13 \to 1)|_{T_0}^{-1}(u)=P_r(d)-1.\]
  Since $(123 \to 13)(F) \cong \mb{P}(I_d(\mc{C}'_u))$, we conclude that
  for a general element \[(\mc{C}'_u,\mc{C}_{u,g}) \in (12 \to 1)|_{L_r}^{-1}(u),\] $(123 \to 12)|_F^{-1}(\mc{C}'_u,\mc{C}_{u,g})$
  parametrizes all degree $d$ surfaces containing $\mc{C}_{u,g}$.
  In other words, the generic fiber of $(123 \to 12)|_F$ is isomorphic to $\mb{P}(I_d(\mc{C}_{u,g}))$. Hence, \eqref{de-11} implies that 
  \begin{equation}\label{de-12}
           \dim (12 \to 1)^{-1}(u) \cap L_r=\dim F-\dim \mb{P}(I_d(\mc{C}_{u,g}))=P_r(d)-\dim I_d(\mc{C}_{u,g}),
          \end{equation}
which is equal to $\dim I_d(\mc{C}'_u)-\dim I_d(\mc{C}_{u,g})$.

We will now prove the second equality. This is equivalent to proving $\dim \mb{P}(I_d(\mc{C}_{u,g}))$ equals $\dim \mb{P}(I_d({C}_g))$. By the upper semicontinuity
 of fiber dimension, we know $\dim \mb{P}(I_d(\mc{C}_{u,g})) \ge \dim (I_d(C_g))$. 
 Using the same arguments as before, one can check that 
 \[\dim (12 \to 1)|_{L_r}^{-1}(C_{g,\red})=P_r(d)-1-\dim \mb{P}(I_d(C_g)).\]
 By \eqref{de-12} and the upper semicontinuity of fiber dimension, we then have  
 {\small \[P_r(d)-1-\dim \mb{P}(I_d(C_g))=\dim (12 \to 1)|_{L_r}^{-1}(C_{g,\red}) \le \dim (12 \to 1)|_{L_r}^{-1}(u)=P_r(d)-1-\dim \mb{P}(I_d(\mc{C}_{u,g})).\]}
 This gives the reverse inequality, $\dim \mb{P}(I_d(C_g)) \ge \dim \mb{P}(I_d(\mc{C}_{u,g}))$. Hence, $\dim \mb{P}(I_d(C_g))$ equals $\dim \mb{P}(I_d(\mc{C}_{u,g}))$.
 This proves the theorem.
\end{proof}

   \section{infinitesimal deformation of non-reduced curves}
   In this section we show that the curves studied by Martin-Deschamps and Perrin (see Theorem \ref{a84}) have 
   an interesting deformation theoretic property: there exists an infinitesimal deformation of such 
   a curve such that the associated reduced scheme does not deform  (see Corollary \ref{phe44} and Remark \ref{phe14}).
   To prove this we first give a general criterion under which a $d$-embeddable curve of the form $2C_1+C_2$, for
   $C_1, C_2$ reduced curves, satisfy this property (see Theorem \ref{tr1}). We prove in Corollary \ref{phe44}
   that the curves in Theorem \ref{a84} satisfy this criterion.
   
   \begin{se}
    Let $C_1, C_2$ be two reduced curves in $\p3$ without common components and $X$ a smooth degree $d$ surface in $\p3$ containing $C_1 \cup C_2$ for some $d \ge 2\deg(C_1)+\deg(C_2)+4$. 
    Denote by $C$ the effective divisor in $X$ of the form $2C_1+C_2$. Denote by $P_1$ (resp. $P, P_r$)
    the Hilbert polynomial of $C_1$ (resp. $C$, $C_{\red}$). We use notations as in Notation \ref{de-4}.
   \end{se}

  \begin{defi}
   A \emph{first order} $d$-\emph{embedded curve} is a $d$-embedded curve $(C,X)$ such that 
   $\Ima \beta_C \subset \Ima \rho_{_C}$, where $\beta_C$ and $\rho_{_C}$ are as in the diagram \eqref{phe11}. 
  \end{defi}
  
 Note that $(C,X)$ is first order $d$-embedded if for every first order deformation $C'$ of $C$ there exists a first order deformation $X'$ of $X$ such that $C' \subset X'$.
    Recall the following useful result:
 
  \begin{cor}[{\cite[Corollary $3.8$]{dan7}}]\label{dim4}
  The following holds true: The kernel of $\rho_{_C}$ is isomorphic to $H^0(\mo_X(-C)(d))$ and $\rho_{_C}$ is surjective if and only if $H^1(\mo_X(-C)(d))=0$.
   Moreover, if $H^1(\mo_X(-C)(d))=0$, then $\pr_1(T_{(C,X)}H_{P,Q_d})=H^0(\N_{C|\p3})$.
  \end{cor}
  
  \begin{lem}\label{phd03}
   If $C$ is $d+1$-regular (in the sense of Castelnuovo-Mumford regularity), then $\rho_{_C}$ is surjective. In this case, $(C,X)$ is first order $d$-embedded.
  \end{lem}

  \begin{proof}
   The definition of Castelnuovo-Mumford regularity implies
   that if $C$ is $d+1$-regular, then $H^1(\mo_X(-C)(d))=0$. By Corollary \ref{dim4}, this implies the
   surjectivity of $\rho_{_C}$.  Since $\rho_{_C}$ is surjective, $\Ima \beta_C \subset \Ima \rho_{_C}$. 
   Hence, $C$ is first order $d$-embedded. This proves the lemma.
  \end{proof}

  The following lemma states that if $(C,X)$ is a first order $d$-embedded curve and there exists a first order 
 deformation $C_{\xi}$ of $C$ containing  a first order deformation $C'_{\xi}$ of 
  $C_{\red}$, then $C'_{\xi}$ also contains a first order deformation of $C_1$. The proof of this statement uses the 
  $\mb{R}$-linearity of the Gauss-Manin connection.
  
  \begin{lem}
  If $(C,X)$ is first order $d$-embedded, then 
  \begin{equation}\label{pheeq3}
     \dim \pr_1T_{(C_{\red},C)}H_{P_r,P} \le \dim \pr_2T_{(C_1,C_{\red})}H_{P_1,P_r}
     \end{equation}
  where $\pr_1: T_{(C_{\red},C)} H_{P_r,P} \to T_{C_{\red}} H_{P_r}$ and $\pr_2:T_{(C_1,C_{\red})}H_{P_1,P_r} \to T_{C_{\red}} H_{P_r}$
  are natural projection maps.
  \end{lem}

  \begin{proof}
     Given $(\xi_1,\xi_2) \in T_{(C_{\red},C)}H_{P_r,P}$, we are going to show that there exist $\xi' \in H^0(\N_{C_1|\p3})$ such that $(\xi',\xi_1) \in T_{(C_1,C_{\red})}H_{P_1,P_r}$.
     Since $(C,X)$ is first order $d$-embedded, given a pair $(\xi_1, \xi_2) \in T_{(C_{\red},C)}H_{P_r,P}$, there exists $\xi \in H^0(\N_{X|\p3})$
     such that $\beta_C(\xi_2)=\rho_{_C}(\xi)$ and \[\up{5}{C_{\red}}{C}(\xi_2)=\up{6}{C_{\red}}{C}(\xi_1).\] Using Lemma \ref{nr13} we conclude,
     \[\rho_{C_{\red}}(\xi)=\up{2}{C_{\red}}{C} \circ \rho_{_C}(\xi)=\up{2}{C_{\red}}{C} \circ \beta_C(\xi_2)=\up{1}{C_{\red}}{C} \circ \up{5}{C_{\red}}{C}(\xi_2)=\]\[=\up{1}{C_{\red}}{C} \circ \up{6}{C_{\red}}{C}(\xi_1)=\beta_{C_{\red}}(\xi_1).\] 
     Corollary \ref{pr13} implies $\ov{\nabla}([2C_1+C_2])(\xi)=0$ and $\ov{\nabla}([C_1+C_2])(\xi)=0$. Since the differential $\ov{\nabla}$ is $\mb{R}$-linear, $\ov{\nabla}([C_1])(\xi)=0$
     and $\ov{\nabla}([C_2])(\xi)=0$.     
     But $\deg(C_1), \deg(C_2)$  and  $\deg(C_{\red})$ are less than $d-4$. 
     Hence by Corollary \ref{pr13}, there exists $\xi' \in H^0(\N_{C_1|\p3})$ such that $\up{1}{C_1}{C_{\red}} \circ \up{6}{C_1}{C_{\red}}(\xi')=\up{2}{C_1}{C_{\red}} \circ \rho_{C_{\red}}(\xi)$. So, 
     \[\up{1}{C_1}{C_{\red}} \circ \up{6}{C_1}{C_{\red}}(\xi')=\up{2}{C_1}{C_{\red}} \circ \rho_{C_{\red}}(\xi)=\up{2}{C_1}{C_{\red}} \circ \beta_{C_{\red}}(\xi_1)=\]\[=\up{1}{C_1}{C_{\red}} \circ \up{5}{C_1}{C_{\red}}(\xi_1).\]
     Using Lemma \ref{nr13}$(6)$, we can conclude that $\up{1}{C_1}{C_{\red}}$ is injective. So, \[\up{6}{C_1}{C_{\red}}(\xi')=\up{5}{C_1}{C_{\red}}(\xi_1).\] 
     Hence, $(\xi',\xi_1) \in T_{(C_1,C_{\red})}H_{P_1,P_r}$. 
     This completes the proof of the lemma.
  \end{proof}

  \begin{prop}\label{phe27}
   If $(C,X)$ is first order $d$-embedded, then the fiber over $C_{\red}$ to the morphism $(12 \to 1):H_{P_r,P} \to H_{P_r}$ is smooth at the point $(C_{\red},C)$.
   In particular, \[\dim (12 \to 1)^{-1}(C_{\red})=\dim \ker \up{5}{C_{\red}}{C}.\]
  \end{prop}

  \begin{proof}
  Lemma \ref{nr13}$(3)$ implies that $\up{6}{C_{\red}}{C}$ is injective. So, the tangent space to the fiber
   at $(C_{\red},C)$, $T_{(C_{\red},C)}((12 \to 1)^{-1}(C_{\red}))$, is isomorphic to the kernel of $\up{5}{C_{\red}}{C}$. Lemma \ref{nr13}$(5,6)$ imply that $\up{1}{C_{\red}}{C}$ and $\beta_C$ are injective.
   So, $\dim\ker \up{5}{C_{\red}}{C}$ equals $\dim \beta_C (\ker \up{5}{C_{\red}}{C})$ and \[ \ker \up{5}{C_{\red}}{C} = \ker \up{1}{C_{\red}}{C} \circ \up{5}{C_{\red}}{C} = \ker \up{2}{C_{\red}}{C} \circ \beta_C.\]
   Now, $\beta_C(\ker (\up{2}{C_{\red}}{C} \circ \beta_C))=\ker \up{2}{C_{\red}}{C} \cap \Ima \beta_C$.
  Since $\Ima \beta_C \subset \Ima \rho_{_C}$ ($(C, X)$ is first order $d$-embedded), we have \[\ker \up{2}{C_{\red}}{C} \cap \Ima \beta_C \subset \ker \up{2}{C_{\red}}{C} \cap \Ima \rho_{_C} = \rho_{_C}( \ker \up{2}{C_{\red}}{C} \circ \rho_{_C}).\]
  Since $\up{2}{C_{\red}}{C} \circ \rho_{_C}=\rho_{C_{\red}}$ (Lemma \ref{nr13}),  by Corollary \ref{dim4} we have \[\ker \rho_{_C} \cong H^0(\mo_X(-C)(d)) \, \mbox{ and } \, \ker \up{2}{C_{\red}}{C} \circ \rho_{_C} \cong H^0(\mo_X(-C_{\red}(d))).\]
  Therefore,
  \begin{equation}\label{pheeq2}
  \dim \ker \up{5}{C_{\red}}{C} \le \dim \ker (\up{2}{C_{\red}}{C} \circ \rho_{_C})-\dim \ker \rho_{_C}=h^0(\I_{C_{\red}}(d))-h^0(\I_C(d)).
  \end{equation}
      Conversely, Theorem \ref{phe42} implies that  \[\dim I_d(C_{\red})-\dim I_d(C) \le \dim (12 \to 1)^{-1}(C_{\red}) \le \]\[\le \dim T_{(C_{\red},C)} (12 \to 1)^{-1}(C_{\red}) = \dim \ker \up{5}{C_{\red}}{C}.\]
  Using the inequality (\ref{pheeq2}) we therefore conclude \[\dim (12 \to 1)^{-1}(C_{\red})=\dim T_{(C_{\red},C)} (12 \to 1)^{-1}(C_{\red})= \dim I_d(C_{\red})-\dim I_d(C).\]
  Hence, the fiber is smooth at the point $(C_{\red},C)$. This proves the proposition.
  \end{proof}
   
   \begin{thm}\label{tr1}
   Consider the following setup:
   \begin{enumerate}
    \item $(C,X)$ as before is first order $d$-embedded and there exists an irreducible component 
   $L$ of $H_P$ containing $C$ such that a general element $C_g$ of $L$ is $d$-embeddable
   and the Hilbert polynomial of $C_{g,\red}$ is $P_r$,
   \item let $L_r$ be the irreducible component 
    of $H_{P_r,P}$, as in  Proposition \ref{phe8}, containing the point corresponding to $(C_{\red},C)$ and mapping surjectively to $L$ via the morphism $(12 \to 2)$,
    \item there exists an irreducible component, say $L_0$ of $H_{P_1,P_r}$ such that $L_0$ is smooth at $(C_1,C_{\red})$
    and $\pr_2(L_0)_{\red}$ is contained in $(12 \to 1)(L_r)_{\red}$, where $\pr_2$ is the natural
    projection map from $H_{P_1,P_r}$ to $H_{P_r}$.
   \end{enumerate}
    Then $L_r$ is smooth at $(C_{\red},C)$ and $\dim \pr_2 L_0=\dim (12 \to 1) L_r$.
   \end{thm}

  \begin{proof}
   By Lemma \ref{nr13}$(3)$, we have $\up{6}{C_1}{C_{\red}}$ is injective.
       Since $L_0$ is smooth at $(C_1,C_{\red})$,
     \begin{equation}\label{pheeq4}
     \dim L_0=\dim T_{(C_1,C_{\red})}L_0=\dim \pr_2T_{(C_1,C_{\red})}L_0.
     \end{equation}
     
    By Proposition \ref{phe27}, we have $\dim \ker \up{5}{C_{\red}}{C}=\dim (12 \to 1)^{-1}(C_{\red})$. Hence,
     \begin{eqnarray}\label{pheeq5}
&\dim T_{(C_{\red},C)}L_r= \dim \pr_1T_{(C_{\red},C)}L_r+\dim \ker \up{5}{C_{\red}}{C}= \nonumber \\
 &=\dim \pr_1T_{(C_{\red},C)}L_r+\dim (12 \to 1)^{-1}(C_{\red}).
 \end{eqnarray}
 By Proposition \ref{phe8}, a general element $C'_g \in (12 \to 1)(L_r)$ is reduced, hence $d$-regular by Theorem \ref{cas1}. 
 This implies $\dim I_d(C_{\red})=\dim I_d(C'_g)$. Using  Theorem \ref{phe42}, we then conclude that 
     $\dim (12 \to 1)^{-1}(C_{\red})=\dim (12 \to 1)^{-1}(C'_g)$. 
     
 Now, $\dim L_0=\dim \pr_2L_0$ because 
     the fiber of $\pr_2$ is zero dimensional (there are only finitely many curves with Hilbert polynomial $P_1$ in $C_{\red}$). Finally, we have
     \[\dim \pr_1T_{(C_{\red},C)}L_r+\dim (12 \to 1)^{-1}(C'_g) \stackrel{(\ref{pheeq3})}{\le} \dim \pr_2T_{(C_1,C_{\red})}L_0+\dim (12 \to 1)^{-1}(C'_g)=\]
     \[\stackrel{(\ref{pheeq4})}{=}\dim L_0+\dim L_r-\dim (12 \to 1)(L_r).\]
     Using (\ref{pheeq5}) we have $\dim T_{(C_{\red},C)}L_r-\dim L_r \le \dim \pr_2L_0 - \dim (12 \to 1)(L_r)$.
 By the hypothesis, $\dim \pr_2L_0  \le \dim (12 \to 1)(L_r)$. Hence, $\dim T_{(C_{\red},C)}L_r \le \dim L_r$. Since the dimension of the tangent space of a scheme at a point is 
     at least equal to the dimension of the scheme at that point, we have \[\dim T_{(C_{\red},C)}L_r=\dim L_r \, \mbox{ and } \, \dim \pr_2L_0=\dim (12 \to 1)(L_r).\]
     This proves the theorem.    
          \end{proof}
          
    We now give an example of $C_1$ and $C$ satisfying the last hypothesis of Theorem \ref{tr1}. This will be used in the proof of Corollary \ref{phe44}.

          \begin{lem}\label{phd04}
          Let $l$ be a line and $C_2$ a smooth coplanar curve (on a plane containing $l$) such that $\#\{l \cap C_2\}<\infty$. Denote by $P_1$ (resp. $P_r$) the Hilbert polynomial of $l$ (resp. $l \cup C_2$).
           Then $H_{P_1,P_r}$ is smooth at $(l,l \cup C_2)$.
          \end{lem}

          \begin{proof}
           Consider the projection map $\pr_1:H_{P_1,P_r} \to H_{P_1}$. 
           Note that the dimension of the fiber over a line $l_0$ to $\pr_1$ is equal to $\dim \mb{P}(H^0(\mo_{\mb{P}^2}(\deg(C_2))))+1$, where the first term is the dimension of the space of 
   degree $\deg(C_2)$ curves on a plane containing $l_0$ and the second term is the dimension of the space of planes in $\p3$ containing $l_0$. 
   Since $l_0 \in H_{P_1}$ is arbitrary and $\pr_1$ is surjective, $\dim H_{P_1,P_r}=\dim \mb{P}(\mo_{\mb{P}^2}(\deg(C_2)))+1+\dim H_{P_1}$.
   
   Since $l$ and $l \cup  C_2$ are complete intersection curves in $\p3$, \[\N_{l \cup C_2} \cong \mo_{l\cup C_2}(1) \oplus \mo_{l\cup C_2}(\deg(C_2)+1)\] and by \cite[Ex. III.$5.5$]{R1} the
   natural morphism from $H^0(\mo_{l\cup C_2}(k))$ to $H^0(\mo_l(k))$ is surjective for all $k \in \mb{Z}$. In particular, $\up{5}{l}{l \cup C_2}$ is surjective.
   Hence, the dimension of the tangent space $T_{(l,l\cup C_2)}H_{P_1,P_r}$ is equal to 
   \[h^0(\N_{l|\p3})+\dim \ker \up{5}{l}{l \cup C_2}=h^0(\N_{l|\p3})+(h^0(\mo_{l\cup C_2}(1))+  h^0(\mo_{l\cup C_2}(t)))-\]\[-(h^0(\mo_{l}(1))+ h^0(\mo_{l}(t))),
   \mbox{ where } t=\deg(C_2)+1.\]
   
   The homogeneous ideal of $l\cup C_2$ (resp. $l$) contains one (resp. two) linearly independent linear polynomials. So, $h^0(\mo_{l\cup C_2}(1))=h^0(\mo_{l}(1))+1$.
   Note then that the dimension of $\ker \up{5}{l}{l \cup C_2}$ is equal to $h^0(\mo_{l\cup C_2}(\deg(C_2)+1))-h^0(\mo_{l}(\deg(C_2)+1))+1$.
   Since $l \cup C_2$ and $l$ are $t$-regular (Theorem \ref{cas1}), one can use their Hilbert polynomials to prove that 
   \[\ker \up{5}{l}{l \cup C_2}=(\deg(C_2)+1)\deg(C_2)-\rho_a(l \cup  C_2)+1=\frac{\deg(C_2)(\deg(C_2)+3)}{2}+1\]
   which is equal to $\dim \mb{P}(\mo_{\mb{P}^2}(\deg(C_2)))+1$.
   Since $H_{P_1}$ is smooth and irreducible (Hilbert scheme parametrizing lines in $\p3$), we have
   \[\dim T_{(l, l \cup C_2)}H_{P_1,P_r}=h^0(\N_{l|\p3})+\frac{\deg(C_2)(\deg(C_2)+3)}{2}+1=\]\[=\dim H_{P_1}+\dim \mb{P}(\mo_{\mb{P}^2}(\deg(C_2)))+1=\dim H_{P_1,P_r}.\]
   This proves the lemma.
          \end{proof}

 We now recall the classical example of Martin-Deschamps and Perrin.
  
\begin{note}\label{a7}
 Let $a, d$ be positive integers, $d \ge 5$ and $a>0$. Let $X$ be a smooth projective surface of degree $d$ containing a line $l$ and a smooth coplanar curve $C_2$ of degree $a$.
 Let $C$ be a divisor of the form $2l+C_2$ in $X$. Denote by $P$ (resp. $P_r$) the Hilbert polynomial of $C$ (resp. $C_{\red}$).
\end{note}

\begin{thm}[Martin-Deschamps and Perrin \cite{mar1}]\label{a84}
There exists an irreducible component, say $L$ of $H_P$ such that a general curve $D \in L$ is a divisor in a smooth degree $d$ surface, of the form $2l'+C_2'$ where $l', C_2'$ are coplanar curves with $\deg(l')=1$
and $\deg(C_2')=a$. Furthermore, $L$ is generically non-reduced.
\end{thm}

\begin{proof}
 The theorem follows from \cite[Proposition $0.6$, Theorems $2.4, 3.1$]{mar1}.
\end{proof}

\begin{defi}
 Given a scheme $X$ and a point $x \in X$, we say that $x$ is \emph{weakly general} if $x$ is contained in an 
 unique irreducible component of $X$.
\end{defi}

In the following corollary we see that the examples in Theorem \ref{a84} satisfy the hypotheses in Theorem \ref{tr1}.
          \begin{cor}\label{phe44}
          Notations as in \ref{a7} and Theorem \ref{a84}. Let $L_r$ be the irreducible  component
           of $H_{P_r,P}$, as in Proposition \ref{phe8}, mapping surjectively to $L$.  
           If $(C_{\red},C)$ correspond to a weakly general point on $H_{P_r,P}$, then $L_r$ is smooth at this point.       
       In particular, there exists $\xi_0 \in H^0(\N_{C|\p3})$ such that  $\up{5}{C_{\red}}{C}(\xi_0) \not\in \Ima \up{6}{C_{\red}}{C}$.
          \end{cor}

  \begin{proof}
  Using Theorem \ref{tr1} it suffices to show that $(C,X)$ is first order $d$-embedded, there exists an irreducible component, $L_0$ of $H_{P_1,P_r}$ such that
   $\pr_2(L_0)_{\red} \subset (12 \to 1)(L_r)_{\red}$ and $L_0$ is smooth at $(l,C_{\red})$. 
   
  Now, \cite[Theorem $4.12$]{dan7} implies that the Castelnuovo-Mumford regularity of $C$ is at most
  $d+1$, hence first order $d$-embedded by Lemma \ref{phd03}.  By the definition of $L$, $(12 \to 1)(L_r)_{\red}$
   contains all coplanar curves which are the union of a line and a degree $\deg(C_2)$ coplanar curve. Note that there exists an irreducible sub-variety in $H_{P_r}$ which 
   parametrizes all such curves and there exists
   an irreducible component, say $L_0$, in $H_{P_1,P_r}$ which maps surjectively to this sub-variety.
   Finally, Lemma \ref{phd04} implies $L_0$ is smooth at $(l,C_{\red})$. This proves the first part of the corollary.
   
   Recall,  $L$ is non-reduced at the point corresponding to $C$. Since $L_r$ is reduced, $(12 \to 2)(L_r)$ is reduced 
   at the point corresponding to $C$ (scheme-theoretic image of a reduced scheme is reduced). 
   This means that there exists $\xi_0 \in H^0(\N_{C|\p3})$ not contained in the image
   of the projection map from $T_{(C_{\red},C)}H_{P_r,P}$ to $T_CH_P$. In particular,
   $\up{5}{C_{\red}}{C}(\xi_0) \not\in \Ima \up{6}{C_{\red}}{C}$. This proves the rest of the corollary.
  \end{proof}
  
  \begin{rem}\label{phe14}
   Note that $\up{5}{C_{\red}}{C}(\xi_0) \not\in \Ima \up{6}{C_{\red}}{C}$, in Corollary \ref{phe44}, means that \emph{there exists
   a first order deformation of $C$ such that $C_{\red}$ does not deform}.  
  \end{rem}

\section{Extension of curves and induced non-reducedness}
 In the previous section we gave examples of $d$-embeddable curves under which there
 exists an infinitesimal deformation of the curve such that the associated reduced scheme does not deform.
 In this section we introduce ``extension of curves''. We observe that if $D$ and $C$ are $d$-embeddable curves
 with $D$ an extension of  $C$ and there exists an infinitesimal deformation 
 of $C$ such that $C_{\red}$ does not deform, then there exists an infinitesimal deformation of $D$ such that $D_{\red}$ does not deform (see Proposition \ref{phd06}).
 We use this to see under certain conditions, the Hilbert scheme containing $D$ is non-reduced (see Theorem \ref{phd05}).
 
 \begin{defi}\label{pr5}
 Let $(C,X)$ and $(D,X)$ be two $d$-embedded curves. We say that 
 $(D,X)$ is a \emph{simple extension of  } $(C,X)$ if it satisfies the following conditions:
 \begin{enumerate}
  \item $D^c:=D-C$ is of the form $nC'$ for some positive integer $n$, $C'\subset X$ reduced curve and $C' \cap C_{\red}$ consists of finitely many points,
  \item The image of $\up{6}{C}{D}$ is contained in the image of $\up{5}{C}{D}$.
 \end{enumerate}
 We say that $(D,X)$ is an \emph{extension of  } $(C,X)$ if there exists a sequence \[C=C_0 \le C_1 \le ... \le C_n=D\] of effective divisors on $X$ such that
 $(C_{i+1},X)$ is a simple extension of  $(C_i,X)$ for $i \ge 0$.
\end{defi}

 \begin{prop}\label{phd06}
Let $(C,X)$ be a $d$-embedded curve for some $d \ge \deg(C)+4$. 
Suppose $\xi_0 \in H^0(\N_{C|\p3})$ be such that $\up{5}{C_{\red}}{C}(\xi_0) \not\in \Ima \up{6}{C_{\red}}{C}$.
Let $(D,X)$ be an extension of $(C,X)$ and $\deg(D) \le d-4$.
If $(D,X)$ is first order $d$-embedded, then there exists $\xi \in H^0(\N_{D|\p3})$ such that $\up{5}{D_{\red}}{D}(\xi) \not\in \Ima \up{6}{D_{\red}}{D}$.
\end{prop}

 \begin{proof}
  Suppose that there exists a chain $C=C_0 \le C_1 \le ... \le C_n=D$ of effective divisors on $X$ such that 
  $(C_{i+1},X)$ is a simple extension of  $(C_i,X)$ for $i=0,...,n-1$. We first show that if $(D,X)$ is first order $d$-embedded, then 
 each $(C_i,X)$ for $i=0,...,n$ is first order $d$-embedded. Suppose not i.e., there exists some $i \in \{0,...,n\}$ and $\xi'_i \in H^0(\N_{C_{i}|\p3})$ 
 such that $\beta_{C_i}(\xi'_i) \not\in \Ima \rho_{_{C_i}}$. 
 Since $(C_{i+1},X)$ is a simple extension of  $(C_{i},X)$, 
  there exists $\xi'_{i+1} \in H^0(\N_{C_{i+1|\p3}})$ such that $\up{5}{C_i}{C_{i+1}}(\xi'_{i+1})=\up{6}{C_i}{C_{i+1}}(\xi'_i)$. Now, 
  \[\up{2}{C_i}{C_{i+1}} \circ \beta_{C_{i+1}}(\xi'_{i+1})=\up{1}{C_i}{C_{i+1}} \circ \up{5}{C_i}{C_{i+1}}(\xi'_{i+1})=\up{1}{C_i}{C_{i+1}} \circ \up{6}{C_i}{C_{i+1}}(\xi'_i)\]
  which by Lemma \ref{nr13} is equal to   $\beta_{C_i}(\xi'_i)$.
  Since $\rho_{_{C_i}}=\up{2}{C_i}{C_{i+1}} \circ \rho_{_{C_{i+1}}} $, if $\beta_{C_{i+1}}(\xi'_{i+1}) \in \Ima \rho_{_{C_{i+1}}}$, then 
  \[\beta_{C_i}(\xi'_i)=\up{2}{C_i}{C_{i+1}} \circ \beta_{C_{i+1}}(\xi'_{i+1}) \in \up{2}{C_i}{C_{i+1}}(\Ima \rho_{_{C_{i+1}}})=\Ima \rho_{_{C_i}}\] where the last equality follows from 
  Lemma \ref{nr13}. But this contradicts the assumption on $\xi'_i$.
  Hence, $\beta_{C_{i+1}}(\xi'_{i+1}) \not\in \Ima \rho_{_{C_{i+1}}}$.  Proceeding recursively, we get a contradiction to $\Ima \beta_D \subset \Ima \rho_{_D}$. 
  Therefore, for each $i=0,...,n$, $(C_i,X)$ is first order $d$-embedded.
 
 Since $(C_{i},X)$ is a simple extension of  $(C_{i-1},X)$,  there exists $\xi_i \in H^0(\N_{C_i|\p3})$ such that 
 \[\up{5}{C_{i-1}}{C_{i}}(\xi_i)=\up{6}{C_{i-1}}{C_{i}}(\xi_{i-1})\, \, \mbox{ for all } \, i \in \{1,...,n\}.\] 
 Using the diagram \ref{phe11} we observe that {\small \[\rho_{C_i}^{-1}(\beta_{C_i}(\xi_i)) \subset (\up{2}{C_{i-1}}{C_i} \circ \rho_{C_i})^{-1} \circ (\up{2}{C_{i-1}}{C_i} \circ \beta_{C_i}(\xi_i))=(\up{2}{C_{i-1}}{C_i} \circ \rho_{C_i})^{-1} \circ (\up{1}{C_{i-1}}{C_i} \circ \up{5}{C_{i-1}}{C_i}(\xi_i)) \subset\]
 \[\subset (\up{2}{C_{i-1}}{C_i} \circ \rho_{C_i})^{-1}(\up{1}{C_{i-1}}{C_i} \circ \up{6}{C_{i-1}}{C_i}(\xi_{i-1}))=\rho_{C_{i-1}}^{-1}(\beta_{C_{i-1}}(\xi_{i-1}))\]}
 where the last equality follows from Lemma \ref{nr13}. Hence by Corollary \ref{pr13}, 
 \[\ov{\nabla}([C_i])(t)=0=\ov{\nabla}([C_{i-1}])(t)\, \, \, \mbox{ for all } t \in \rho_{C_i}^{-1}(\beta_{C_i}(\xi_i)).\]
 In other words, $\ov{\nabla}([C_i])(t)=0$ for all $t \in \rho_{_D}^{-1}(\beta_D(\xi_n)) \subset \rho_{_C}^{-1}(\beta_C(\xi_0))$ and $i=0,...,n$.
 
 Let $C_{i}=C_{i-1}+a_{i-1}C'_{i-1}$ for some $a_{i-1}$ positive integer and $C'_{i-1}$ reduced curve for $i=1,...,n$.
 Then we conclude  
 \[0=\ov{\nabla}([C_i]-[C_{i-1}])(t)=a_{i-1}\ov{\nabla}([C'_{i-1}])(t) \, \mbox{ for all } t \in \rho_{_D}^{-1}(\beta_D(\xi_n)).\]
 Since $C'_{i-1} \cap C_{i-1,\red}$ consists of finitely many points, $C_{i,\red}-C_{i-1,\red}=C'_{i-1}$.
 Hence, 
 \begin{equation}\label{de-20}
  0=\ov{\nabla}([C'_{i-1}])(t) = \overline{\nabla}([C_{i,\red}]-[C_{i-1,\red}])(t) \, \mbox{ for all } t \in \rho_{_D}^{-1}(\beta_D(\xi_n)).
 \end{equation}
  Suppose now that $\up{5}{D_{\red}}{D}(\xi_n) \in \Ima \up{6}{D_{\red}}{D}$. Then 
 {\small \[\rho_{D}^{-1}(\beta_{D}(\xi_n)) \subset (\up{2}{D_{\red}}{D} \circ \rho_{D})^{-1} \circ (\up{2}{D_{\red}}{C_i} \circ \beta_{D}(\xi_n))=\]
 \[=(\up{2}{D_{\red}}{D} \circ \rho_{D})^{-1} \circ (\up{1}{D_{\red}}{D} \circ \up{5}{D_{\red}}{D}(\xi_n)) \subset (\up{2}{D_{\red}}{D} \circ \rho_{D})^{-1}(\Ima (\up{1}{D_{\red}}{D} \circ \up{6}{D_{\red}}{D})).\]}
 Then  Corollary \ref{pr13} implies $\overline{\nabla}([D_{\red}])(t)=0$ for all $t \in {\rho_{D}}^{-1}(\beta_{D}(\xi_n))$. 
 Substituting this in \eqref{de-20}, we have $\ov{\nabla}[C_{i,\red}](t)=0$ for all $t \in \rho_{_D}^{-1}(\beta_D(\xi_n))$ and $i=0,...,n$.
 Corollary \ref{pr13} implies that for a given $t \in \rho_{_D}^{-1}(\beta_D(\xi_n))$, there exists $\xi'_{0}(t) \in H^0(\N_{C_{\red}|\p3})$ such that 
 $\rho_{_{C_{\red}}}(t)=\beta_{C_{\red}}(\xi'_{0}(t))$. Since $\rho_{_D}^{-1}(\beta_D(\xi_n)) \subset \rho_{_C}^{-1}(\beta_C(\xi_0))$, Lemma \ref{nr13} implies that
 \[\up{1}{C_{\red}}{C} \circ \up{6}{C_{\red}}{C}(\xi'_{0}(t))=\up{2}{C_{\red}}{C} \circ \rho_{C}(t)=\up{2}{C_{\red}}{C} \circ \beta_{C}(\xi_{0})=\up{1}{C_{\red}}{C} \circ \up{5}{C_{\red}}{C}(\xi_{0}).\]
  By Lemma \ref{nr13}$(6)$, $\up{1}{C_{\red}}{C}$ is injective. Hence, $\up{6}{C_{\red}}{C}(\xi'_{0}(t))=\up{5}{C_{\red}}{C}(\xi_{0})$, contradicting 
  the assumption on $\xi_0$. Hence, $\up{5}{D_{\red}}{D}(\xi_n) \not\in \Ima \up{6}{D_{\red}}{D}$.
 This proves the proposition.
\end{proof}

The following theorem states that such an infinitesimal deformation $\xi$ gives rise to a singularity on the corresponding Hilbert scheme.

\begin{thm}\label{phd05}
 Let $(C,X)$ be as in Proposition \ref{phd06} and $(D,X)$ an extension of $(C,X)$. Denote by $P_{D_r}$ (resp. $P_D$) the Hilbert polynomial of $D_{\red}$ (resp. $D$).
 Suppose further that $(D,X)$ satisfies the following conditions:
 \begin{enumerate}
  \item $\deg(D) \le d-4$,
  \item $(D,X)$ is first order $d$-embedded,
  \item the point $o_D \in H_{P_D}$ corresponding to $D$ is weakly general, and for the unique irreducible component 
  $L$ of $H_{P_D}$ containing $o_D$, the Hilbert polynomial of $D_{g,\red}$ for a general curve $D_g \in L$, is equal to $P_{D_r}$.
   \end{enumerate}
Then $L$ is singular at $o_D$. Furthermore, if a general element of $L$ is $d$-embeddable, then $L$ is non-reduced at $o_D$.
\end{thm}

\begin{proof}
 Replace in Notation \ref{de-4}, $P_r$ by $P_{D_r}$, $P$ by $P_D$.
 By Proposition \ref{phe8}, there exists an irreducible component $L_r$ of $H_{P_{D_{r}},P_D}$ 
 such that $(12 \to 2):L_{r,\red} \to L_{\red}$ is surjective and $(12 \to 2)^{-1}(o_D)$ are points corresponding to pairs $(D',D)$ 
 with $D_{\red}=D'_{\red} \subset D'$.
 Since $D'$ has the same Hilbert polynomial as $D_{\red}$, we conclude $D'=D_{\red}$.
 In particular, $(12 \to 2)^{-1}(o_D)$ is the point corresponding to the pair $(D_{\red},D)$.
 
 Proposition \ref{phd06} implies $h^0(\N_{D|\p3})>\dim \pr_2T_{(D_{\red},D)}L_r$. By the injectivity of $\up{6}{D_{\red}}{D}$ (Lemma \ref{nr13}$(3)$),
 $\dim \pr_2T_{(D_{\red},D)}L_r=\dim T_{(D_{\red},D)}L_r$. Since the fiber to $(12 \to 2)|_{L_r}$ is zero dimensional
    over every point (finitely many sub-curves of a fixed Hilbert polynomial in a fixed curve), we have $\dim L_r=\dim L$.
 Using the diagram \eqref{phe11} and the fact that $C$ is weakly general, we then have
     \[\dim T_{o_D}L=h^0(\N_{D|\p3})>\dim \pr_2T_{(D_{\red},D)}L_r=\dim T_{(D_{\red},D)}L_r \ge \dim L_r=\dim L.\]
      This proves the first part of the theorem.
 
 Suppose now that a general element of $L$ is $d$-embeddable. This means that there exists an open subset $U \subset L$
 such that for all $u \in U$, the corresponding curve $\mc{D}_u$ is $d$-embeddable. Assume that $U$ is reduced. Denote by $U':=(12 \to 2)|_{L_r}^{-1}(U)$.
 Then  Lemma \ref{de-13} implies that the for all $u \in U$, $(12 \to 2)|_{L_r}^{-1}$ is a reduced point, corresponding to the pair $(\mc{D}_{u,\red}, \mc{D}_u)$.
In particular, every fiber over $U$ has the same Hilbert polynomial.
Hence the morphism $(12 \to 2)|_{U'}$ is flat (\cite[Theorem III.$9.9$]{R1}) and proper (base change of proper morphisms is proper). 
Since every fiber of $(12 \to 2)|_{U'}$ is smooth of relative dimension zero, \cite[Ex. III. 10.3]{R1} implies that the morphism $(12 \to 2)|_{U'}$ is \'{e}tale.
     Now, \'{e}tale morphisms induce surjection of tangent spaces i.e., $\pr_2:T_{(D_{\red},D)}L_r \to H^0(\N_{D|\p3})$ is surjective.
     But this contradicts the observation $h^0(\N_{D|\p3})>\dim \pr_2T_{(D_{\red},D)}L_r$ before.
     So, $U$ cannot be reduced. This proves the remaining theorem.
\end{proof}

\begin{exa}\label{de-23}
Let $L$ be as in Theorem \ref{a84} and $(C,X)$ a $d$-embedded curve corresponding to a weakly general point on $L$. By
Corollary \ref{phe44}, there exists $\xi_0 \in H^0(\N_{C|\p3})$  such that $\up{5}{C_{\red}}{C}(\xi_0) \not\in \Ima \up{6}{C_{\red}}{C}$.
 Let $(D,X)$ be an extension of  $(C,X)$ and satisfying the conditions of Theorem \ref{phd05}. 
 Denote by $P_D$ the Hilbert polynomial of $D$.
 Then  by Theorem \ref{phd05}, the unique irreducible component of $H_{P_D}$ containing the point $o_D$ corresponding to $D$, is singular at $o_D$.
\end{exa}

For the sake of completeness  we consider the case when a general element of $L$ is not first order $d$-embedded. We see that in this case as well we get non-reducedness of $L$.
\begin{thm}\label{phd20}
Let $D$ be a $d$-embeddable curve for some $d \ge \deg(D)+4$ and $P_D$ be the Hilbert polynomial of $D$.
Let $L$ be an irreducible component of $H_{P_D}$ such that a general point on $L$ correspond to 
a $d$-embeddable but not first order $d$-embedded curve. Then $L$ is generically non-reduced.
 \end{thm}

 \begin{proof}
Let $D_g$ be a $d$-embedded curve corresponding to a general point on $L$.
There exists an irreducible component $L_s$ of $H_{P_D,Q_d}$ mapping surjectively to $L$ and the fiber 
over the point corresponding to $D_g$ is isomorphic to $\mb{P}(I_d(D_g))$.
Denote by $\pr_1$ the natural morphism from $L_s$ to $L$. 
 By the upper-semicontinuity of fiber dimension, there exists a non-empty open subset $U \subset L$ 
 such that for all $u \in U$ and the corresponding curve $\mc{D}_u$, we have  
 \[\dim \pr_1^{-1}(\mc{D}_u)=\dim \mb{P}(I_d(\mc{D}_u))=\dim \mb{P}(I_d(D_g)).\]
In other words, for all $u \in U$, the fiber $\pr_1^{-1}(u)$ is a projective space of a fixed dimension.
This implies that every such fiber has the same Hilbert polynomial. 

Suppose now that $U$ is reduced. Denote by $U':=\pr_1^{-1}(U)$. Then \cite[Theorem III.$9.9$]{R1} implies that $\pr_1|_{U'}$ 
 is flat with smooth fibers.
By \cite[Theorem III.$10.2$]{R1}, we then conclude that $\pr_1|_{U'}$ is a smooth morphism.
 But a smooth morphism $f:V \to W$ satisfies the condition: the induced differential map is surjective on tangent spaces i.e.,
 $df_v(T_vV)=T_{f(v)}W$ for all $v \in V$. Substituting $f$ by $\pr_1|_{U'}$ we get a contradiction to the assumption that $D_g$ is not first order $d$-embedded
 (by definition, $(D_g,X_g)$ is first order $d$-embedded if and only if $\pr_1:T_{(D_g,X_g)}L_s \to H^0(\N_{D|\p3})$ is surjective).
 Hence, $U$ must be non-reduced.
 Furthermore, replacing $U$ by any open subset of $U$, the same arguments imply that there does not exist 
 any reduced open subset of $U$. Therefore, $L$ is generically non-reduced. This proves the theorem.
 \end{proof}


 


 \section{Applications}
 
   In this section, we give a criterion to extend $d$-embeddable curves (see Theorem \ref{de-15}). We use this to produce examples of 
   extension of curves (see Example \ref{de-21}). Finally, we relate the problem of extending curves to finding families of curves 
   passing through a set of points and avoiding another prescribed set of points (see Theorem \ref{de-18}).
   
   \begin{defi}
    Let $(E,X)$ be a $d$-embedded curve and $S \subset E$ a finite set of reduced, closed points. For any $p \in S$, denote by $S_p:=S\backslash \{p\}$.
    Consider the natural morphism:
    {\small \[r^c_{p}:H^0(\N_{E|\p3}) \to H^0(\N_{E|\p3} \otimes \mo_{S_p}) \, \mbox{ and } r_p:H^0(\N_{E|\p3}) \to H^0(\N_{E|\p3} \otimes \mo_p) \to H^0(\N_{X|\p3} \otimes \mo_p)\]}
    where the last morphism $r_p$ arises after pulling back to $\mo_p$ the normal short exact \eqref{sh2}.
    
    We say that $(E,X)$ is \emph{first order} $S$-\emph{free} if for each point $p \in S$, there exists an element $\xi_p \in H^0(\N_{E|\p3})$
    such that $r^c_{p}(\xi_p)=0$ and $r_p(\xi_p) \not=0$.
   \end{defi}

   \begin{rem}
 Note that $(E,X)$ is first order $S$-free implies that for every $p \in S$, there exists a first order 
 deformation $E'_p$ of $E$ (corresponding to $\xi_p$) such that $S_p \times \Spec(\mb{C}[t]/(t^2)) \subset E'_p$ 
 but $p \times \Spec(\mb{C}[t]/(t^2)) \not\subset E'_p$.
 This follows directly from the relation between the ideal of $E'_p$ in $\mb{P}^3 \times \Spec(\mb{C}[t]/(t^2))$ 
 and the corresponding element $\xi_p \in H^0(\N_{E|\p3})$ as described in \cite[Theorem $2.4$]{R3}.
   \end{rem}

 \begin{thm}\label{de-15}
  Let $(C,X)$ be a $d$-embedded curve and $D^c$ an effective divisor of $X$. Suppose $D^c$ is an integral multiple of a reduced divisor of $X$.
  Denote by $r_{_{C}}$ and $r_{_{D^c}}$ the restriction morphisms,
  \[r_{_{C}}:H^0(\N_{C|\p3}) \to H^0(\N_{C|\p3} \otimes \mo_{C.D^c}) \to H^0(\N_{X|\p3} \otimes \mo_{C.D^c}) \, \mbox{ and }\]
   \[r_{_{D^c}}:H^0(\N_{D^c|\p3}) \to H^0(\N_{D^c|\p3} \otimes \mo_{C.D^c}) \to H^0(\N_{X|\p3} \otimes \mo_{C.D^c}),\]
where the first morphism for both $r_{_{C}}$ and $r_{_{D^c}}$ are canonical restrictions and the second morphism arises as the restriction to $\mo_{C.D^c}$
of the normal exact sequence \eqref{sh2}. If $\Ima(r_{_{C}}) \subset \Ima(r_{_{D^c}})$, then $(C \cup D^c,X)$ is a simple extension of $(C,X)$.
 \end{thm}
  
 \begin{proof}
  Denote by $D:=C \cup D^c$. We simply need to prove that if $\Ima(r_{_{C}}) \subset \Ima(r_{_{D^c}})$, then $\Ima \up{6}{C}{D} \subset \Ima \up{5}{C}{D}$.
   Consider the short exact sequence:
  \begin{equation}\label{de-16}
   0 \to \mo_D \xrightarrow{\rho_{_1}} \mo_C \oplus \mo_{D^c} \xrightarrow{\rho_{_2}} \mo_{C.D^c} \to 0,
  \end{equation}
  where, over any open subset $U \subset D$,
  \[\rho_{_1}|_{_U}(f)=(f \mod \I_C(U),f \mod \I_{D^c}(U)) \, \mbox{ and }\, \rho_{_2}|_{_U}(f,g)=(\ov{f}-\ov{g}).\]
  Consider the natural composed morphisms:
  \[\beta':H^0(\N_{C|\p3}) \oplus H^0(\N_{D^c|\p3}) \to H^0(N_{X|\p3} \otimes \mo_C) \oplus H^0(\N_{X|\p3} \otimes \mo_{D^c}) \to H^0(\N_{X|\p3} \otimes \mo_{C.D^c}),\]
  where the first morphism arises from the normal short exact sequence \eqref{sh2} and the second morphism is induced by $\rho_{_2}$.
  Tensoring \eqref{de-16} by $\N_{D|\p3}$ and taking global sections, we get the following diagram:
  \begin{equation}\label{de-17}
   \begin{diagram}
    & &H^0(\N_{C|\p3}) \oplus H^0(\N_{D^c|\p3})&\rTo^{\beta'}&H^0(\N_{X|\p3} \otimes \mo_{C.D^c})\\
    & &\dTo^{\up{6}{C}{D} \oplus \up{6}{D^c}{D}}& &\uTo^r\\
    H^0(\N_{D|\p3})&\rInto^{\up{5}{C}{D} \oplus \up{5}{D^c}{D}}&H^0(\N_{D|\p3} \otimes \mo_C) \oplus H^0(\N_{D|\p3} \otimes \mo_{D^c})&\rTo^{\rho_{_2}'}&H^0(\N_{D|\p3} \otimes \mo_{C.D^c})
   \end{diagram}
  \end{equation}
 where $r$ arises from pulling back \eqref{sh2} to $\mo_{C.D^c}$. We claim that the right hand square is commutative and 
 the restriction of $r$ to $\rho_{_2}'(\Ima (\up{6}{C}{D} \oplus \up{6}{D^c}{D}))$ is injective.
 
 Let $\xi_0 \in H^0(\N_{C|\p3})$ and $\xi_1 \in H^0(\N_{D^c|\p3})$. Then  $\xi_0$ and $\xi_1$ correspond to $\mo_{\p3}$-linear homomorphisms:
 \[\xi_0:\I_C \to \mo_C \, \mbox{ and }\, \xi_1:\I_{D^c} \to \mo_{D^c}\, \mbox{ respectively.}\]
 Note that $C$ and $D^c$ are local complete intersection in $X$. Hence, for any $x \in C.D^c$, there exists a small enough open 
 neighbourhood $U_x$ of $x$ in $\p3$ and $f_x, g_x \in \mo_{\p3}(U_x)$ such that $\I_{C,x}$ (resp. $\I_{D^c,x}$) is generated by 
 $f_x$ and $F_x$ (resp. $g_x$ and $F_x$), where $F_x$ is a regular function defining $X \cap U_x$ in $U_x$. Note that $\I_{D,x}$ (resp. $\I_{C.D,x}$)
 is generated by $F_x$ and $f_xg_x$ (resp. $F_x$, $f_x$ and $g_x$), as an $\mo_{\p3,x}$-module. Then 
 {\small \[\up{6}{C}{D}(\xi_0)_x:\I_{D,x} \hookrightarrow \I_{C,x} \xrightarrow{\xi_{0,x}} \mo_{C,x} \, \mbox{ sends }\, F_x \mbox{ to } \xi_{0,x}(F_x) \mbox{ and } f_xg_x \mbox{ to } \xi_{0,x}(f_xg_x)=g_x\xi_{0,x}(f_x) \, \mbox{ and }\]
 \[\up{6}{D^c}{D}(\xi_1)_x:\I_{D,x}  \hookrightarrow \I_{D^c,x} \xrightarrow{\xi_{1,x}} \mo_{D^c,x} \, \mbox{ sends }\, F_x \mbox{ to } \xi_{1,x}(F_x) \mbox{ and } f_xg_x \mbox{ to } \xi_{1,x}(f_xg_x)=f_x\xi_{1,x}(g_x).\]}
 Since $f_x, g_x \in \I_{C.D,x}$, observe that 
 \[\rho_{_2}' \circ \up{6}{C}{D}(\xi_0)_x:\I_{D,x} \to \mo_{C.D,x}  \mbox{ sends }\, F_x \mbox{ to } \xi_{0,x}(F_x) \mod \I_{C.D,x} \mbox{ and } f_xg_x \mbox{ to } 0 \, \mbox{ and }\]
  \[\rho_{_2}' \circ \up{6}{D^c}{D}(\xi_1)_x:\I_{D,x} \to \mo_{C.D,x}  \mbox{ sends }\, F_x \mbox{ to } \xi_{1,x}(F_x) \mod \I_{C.D,x} \mbox{ and } f_xg_x \mbox{ to } 0.\]
  The claim follows immediately from this description.
  
  By assumption, for any $\xi \in H^0(\N_{C|\p3})$, there exists $\xi' \in H^0(\N_{D^c|\p3})$ (depending on $\xi$) such that $\beta'(\xi,\xi')=0$.
  By the claim, we have $\rho_{_2}'(\up{6}{C}{D}(\xi),\up{6}{D^c}{D}(\xi'))=0$.
  By the exactness of the bottom row of \eqref{de-17}, there exists $\xi_D \in H^0(\N_{D|\p3})$ such that 
  \[\up{5}{C}{D}(\xi_D)=\up{6}{C}{D}(\xi)\, \mbox{ and }\, \up{5}{D^c}{D}(\xi_D)=\up{6}{D^c}{D}(\xi').\]
  Since $\xi \in H^0(\N_{C|\p3})$ is arbitrary, we conclude that  $\Ima \up{6}{C}{D} \subset \Ima \up{5}{C}{D}$.
  This proves the theorem.
 \end{proof}

 \begin{exa}\label{de-21}
  Notations as in Theorem \ref{de-15}. Suppose that the canonical restriction 
  \[H^0(\T_{\p3} \otimes \mo_{D^c}) \to H^0(\T_{\p3} \otimes \mo_{C.D^c})\] is surjective (for example, if $C.D^c$ is a reduced, closed point). 
  Then  $(C \cup D^c,X)$ is a simple extension of $(C,X)$.
  Indeed, by Theorem \ref{de-15}, it suffices to prove that $r_{_{D^c}}$ is surjective. Consider, the following commutative diagram:
  \[\begin{diagram}
   H^0(\T_{\p3} \otimes \mo_{D^c})&\rTo&H^0(\T_{\p3} \otimes \mo_{C.D^c})&\rTo^a&H^0(\N_{X|\p3} \otimes \mo_{C.D^c})\\
   \dTo&\circlearrowleft&\dTo&\circlearrowleft&\dTo^{\mr{id}}\\
   H^0(\N_{D^c|\p3})&\rTo&H^0(\N_{D^c|\p3} \otimes \mo_{C.D^c})&\rTo&H^0(\N_{X|\p3} \otimes \mo_{C.D^c})
    \end{diagram}\]
 where the first two vertical maps (resp. $a$) are induced by the natural morphism from $\T_{\p3} \otimes \mo_{D^c}$ (resp. $\T_{\p3} \otimes \mo_X$)
 to $\N_{D^c|\p3}$ (resp. $\N_{X|\p3}$) and the bottom horizontal row is simply $r_{_{D^c}}$.
 Note that the morphism $a$ is surjective (pull-back functor is right exact). Hence, the top row of the diagram is surjective.
 By the commutativity of the diagram this implies $r_{_{D^c}}$ is surjective. This proves our claim.
  
  In the case $D^c$ is non-singular, then using Grothendieck vanishing theorem and \cite[Theorem II.$8.13$]{R1}, one can also check that 
  if $\deg(D^c)>C.D^c+2\rho_a(D^c)-2$, then $h^1(\T_{\p3} \otimes \mo_{D^c}(-C.D^c))$ vanishes. This implies that the canonical restriction 
   $H^0(\T_{\p3} \otimes \mo_{D^c}) \to H^0(\T_{\p3} \otimes \mo_{C.D^c})$ is surjective. Hence, similarly as before, $(C \cup D^c,X)$
   is a simple extension of $(C,X)$.
   \end{exa}

   \begin{thm}\label{de-18}
    Let $(E,X)$ be a $d$-embedded curve. Suppose that $E$ is integral. Let $S \subset E$ be a finite set of closed points on $E$. If $E$ is 
    first order $S$-free, then for any $d$-embedded curve $(C,X)$ in $X$ satisfying $C.E \subset S$, the curve $C \cup E$ is a simple extension of 
    $C$. In particular, if there exists a $d$-embedded curve $(C,X)$ in $X$ such that $(C \cup E,X)$ is not a simple extension of $(C,X)$, then
    $E$ is not first order $S$-free for any $S$ containing $C.E$.
   \end{thm}

   \begin{proof}
   Let us first consider the case that $E$ is first order $S$-free.
    Since $H^0(\N_{X|\p3} \otimes \mo_p) \cong \mb{C}$ for any $p \in S$, it follows from definition that the natural morphism,
    \[r_{_S}:H^0(\N_{E|\p3}) \to H^0(\N_{E|\p3} \otimes \mo_{S}) \to H^0(\N_{X|\p3} \otimes \mo_{S})  \, \mbox{ is surjective.}\]
    In particular, the sections $\{r_{_S}(\xi_p)\}_{p \in S}$ generate $H^0(\N_{X|\p3} \otimes \mo_S)$.
    Then  Theorem \ref{de-15} implies that for any $d$-embedded curve $(C,X)$ in $X$ satisfying $C.E \subset S$, 
    the curve $(C \cup E,X)$ is a simple extension of 
    $(C,X)$. This proves the first part of the theorem.
    
    We now prove the second part of the theorem, by contradiction. Suppose that $E$ is first order $S$-free for some 
    $S$ containing $C.E$. As in the proof of the first part, this means the natural morphism 
    \[H^0(\N_{E|\p3}) \xrightarrow{r_{_S}} H^0(\N_{X|\p3} \otimes \mo_S) \, \mbox{ is surjective.}\]
    Hence, the composed morphism 
    \[r_{_{C.E}}:H^0(\N_{E|\p3}) \xrightarrow{r_{_S}} H^0(\N_{X|\p3} \otimes \mo_S) \to  H^0(\N_{X|\p3} \otimes \mo_{C.E}) \, \mbox{ is surjective.}\]
   By Theorem \ref{de-15} this implies $(C \cup E,X)$ is a simple extension of $(C,X)$, which contradicts our hypothesis.
   Hence, $E$ cannot be  first order $S$-free for any $S$ containing $C.E$. This proves the theorem.
   \end{proof}

\bibliographystyle{alpha}
 \bibliography{researchbib}

\end{document}